\documentclass[review]{article}

\usepackage{amsthm, amsmath, amssymb}
\usepackage{graphicx}

\newtheorem{theorem}{Theorem}
\newtheorem{corollary}[theorem]{Corollary}
\newtheorem{lemma}[theorem]{Lemma}
\newtheorem{proposition}[theorem]{Proposition}

\DeclareMathOperator{\Circ}{Circ}

\usepackage{authblk}

\begin{document}

\title{Veto Interval Graphs and Variations}
\author{Breeann Flesch \\ Computer Science Division \\ Western Oregon University \\ 345 North Monmouth Ave. \\ Monmouth, OR 97361 \and
Jessica Kawana$^*$, Joshua D. Laison$^*$ \\ Mathematics Department \\
Willamette University \\ 900 State St. \\ Salem, OR 97301 \and
Dana Lapides$^*$ \\ Earth and Planetary Science \\ University of California, Berkeley \\ 307 McCone Hall \\ Berkeley, CA 94720-4767 \and
Stephanie Partlow\footnote{Funded by NSF Grant DMS 1157105} \\ Mathematics Department \\
Woodburn Wellness, Business And Sports School \\ 1785 N Front St \\ Woodburn, OR 97071 \and 
Gregory J. Puleo \\ Department of Mathematics and Statistics \\ College of Sciences and Mathematics \\ Auburn University \\ 221 Parker Hall \\ Auburn, Alabama 36849}

\maketitle
\newpage

\begin{abstract}
We introduce a variation of interval graphs, called veto interval (VI) graphs.  A VI graph is represented by a set of closed intervals, each containing a point called a veto mark.  The edge $ab$ is in the graph if the intervals corresponding to the vertices $a$ and $b$ intersect, and neither contains the veto mark of the other.  We find families of graphs which are VI graphs, and prove results towards characterizing the maximum chromatic number of a VI graph.  We define and prove similar results about several related graph families, including unit VI graphs, midpoint unit VI (MUVI) graphs, and single and double approval graphs.  We also highlight a relationship between approval graphs and a family of tolerance graphs.
\end{abstract}
\bigskip

\bigskip


\section{Introduction} \label{introduction-section}

An \textit{\textbf{interval representation}} of a graph $G$ is a set of intervals $S$ on the real line and a bijection from the vertices of $G$ to the intervals in $S$, such that for any two vertices $a$ and $b$ in $G$, $a$ and $b$ are adjacent if and only if their corresponding intervals intersect.  A graph is an \textit{\textbf{interval graph}} if it has an interval representation.  Interval graphs were introduced by Haj{\"{o}}s in \cite{Haj}, and were then characterized by the absence of asteroidal triples and induced cycles larger than three by Lekkerkerker and Boland in 1962 \cite{Bollek}.  An \textit{\textbf{asteroidal triple}} in $G$ is a set $A$ of three vertices such that for any two vertices in $A$ there is a path within $G$ between them that avoids all neighbors of the third.  Interval graphs have been extensively studied and characterized, and fast algorithms for finding the clique number, chromatic number, and other graph parameters have been developed \cite{Gol}.  Furthermore, many variations of interval graphs, including interval $p$-graphs, interval digraphs, circular arc graphs, and probe interval graphs, have been introduced and investigated \cite{BroFliLun2, DasRoySenWes, HelHua, McMWanZha}.

We define the following variations of interval graphs.  A \textit{\textbf{veto interval}} $I(a)$ on the real line, corresponding to a vertex $a$, is a closed interval $[a_l,a_r]$ together with a \textit{\textbf{veto mark}} $a_v$ with $a_l<a_v<a_r$.  The numbers $a_l$ and $a_r$ are the \textit{\textbf{left endpoint}} and \textit{\textbf{right endpoint}} of $I(a)$, respectively.  We denote a veto interval as an ordered triple $I(a)=(a_l,a_v,a_r)$.  A \textit{\textbf{veto interval representation}} of a graph $G$ is a set of veto intervals $S$ and a bijection from the vertices of $G$ to the veto intervals in $S$, such that for any two vertices $a$ and $b$ in $G$, $a$ and $b$ are adjacent if and only if either $a_v<b_l<a_r<b_v$ or $b_v<a_l<b_r<a_v$.  In other words, $a$ and $b$ are adjacent if and only if their corresponding intervals intersect and neither contains the veto mark of the other.  In this case we say that $I(a)$ and $I(b)$ are \textit{\textbf{adjacent}}.  If  $b_v<a_l<b_r<a_v$ we say that $b$ intersects $a$ \textit{\textbf{on the left}}, and if $a_v<b_l<a_r<b_v$ we say that $b$ intersects $a$ \textit{\textbf{on the right}}.  A graph $G$ is a \textit{\textbf{veto interval (VI) graph}} if $G$ has a veto interval representation.

If the intervals in $S$ are all the same length, then $S$ is a \textit{\textbf{unit}} veto interval representation, and the corresponding graph $G$ is a \textit{\textbf{unit}} veto interval (UVI) graph.  If no interval in $S$ properly contains another, then $S$ is a \textit{\textbf{proper}} veto interval representation, and the corresponding graph $G$ is a \textit{\textbf{proper}} veto interval (PVI) graph.  If every interval in $S$ has its veto mark at its midpoint, then $S$ is a \textit{\textbf{midpoint}} veto interval representation, and the corresponding graph $G$ is a \textit{\textbf{midpoint}} veto interval (MVI) graph.  We similarly abbreviate midpoint unit veto interval graphs and midpoint proper veto interval graphs as MUVI graphs and MPVI graphs, respectively.  In Section~\ref{unit-midpoint-section}, we show that a graph is a PVI graph if and only if it's an UVI graph, but not all MPVI graphs are MUVI graphs.

If we instead put a directed edge from $a$ to $b$ if $a$ intersects $b$ on the left, we say that $G$ is a \textit{\textbf{directed veto interval graph}}.  Note that a directed veto interval graph is an orientation of the veto interval graph with the same veto interval representation.  Also note that different veto interval representations may yield the same veto interval graph but different directed veto interval graphs.  We use facts about directed veto interval graphs to prove results about the underlying veto interval graphs.

\begin{lemma}\label{directed-cycles}
If $G$ is a directed veto interval graph with directed path $a_1 a_2 \ldots a_k$, $k \geq 3$, then $a_1$ and $a_k$ are not adjacent.
\end{lemma}

\begin{proof}
Since $I(a_1)$ intersects $I(a_2)$ on the left, $a_{1_r}<a_{2_v}<a_{3_l}$.  Likewise $a_{1_r}<a_{k_l},$ so $I(a_1)$ and $I(a_k)$ are disjoint.
\end{proof}

\begin{corollary}
Veto interval graphs are triangle-free.
\end{corollary}

\begin{proof}
In any orientation of a triangle, there is a directed path of length 2.  By Lemma~\ref{directed-cycles}, the graph cannot contain the third edge.
\end{proof}

\begin{corollary}
The complete graph $K_n$ is not a veto interval graph for $n \geq 3$. \hfill $\openbox$
\end{corollary}

\begin{lemma} \label{contained}
Let $a$ and $b$ be vertices of a veto interval graph $G$, and $I(a)$ and $I(b)$ be their veto intervals in a veto interval representation of $G$. If $I(a)$ is contained in $I(b)$, then $a$ and $b$ are not adjacent.
\end{lemma}

\begin{proof}
Assume $I(a)$ is contained in $I(b)$. Then, $b_l \leq a_l<a_v<a_r \leq b_r$. Thus by definition, $a$ and $b$ are not adjacent.
\end{proof}

Note that every induced subgraph of a VI graph is also a VI graph, since given a VI representation $R$ of a graph $G$ with vertex $v$, $R-I(v)$ is a VI representation of $G-v$.

We say that a \textit{\textbf{marked point}} of a veto interval representation $S$ is a left endpoint, veto mark, or right endpoint of some veto interval in $S$.

\begin{lemma}\label{distinct_points}
If a graph $G$ has a veto interval representation $S$, then $G$ has a VI-representation $T$ in which the set of all marked points are distinct.  Furthermore, if $S$ is unit, proper, and/or midpoint, then such a VI-representation $T$ exists that is unit, proper, and/or midpoint also.
\end{lemma}

\begin{proof}
Let $G$ be a veto interval graph and $S$ be a veto interval representation of $G$.  Sort the set of marked points in $S$ in increasing order, $x_1$ through $x_k$.  Suppose $x_i$ is the smallest $x$-value shared by multiple marked points in $S$.  We will construct a new VI-representation $S'$ of $G$.  Let $\varepsilon$ be a distance smaller than the difference between any pair of distinct marked points in $S$.  First, if there is a veto interval with right endpoint $a_r$ at $x_i$, we move $I(a)$ to the right by $\varepsilon$ to make $S'$.  If another interval $I(b)$ has a left endpoint at $x_i$, then $I(a)$ and $I(b)$ are adjacent in $S$ and in $S'$.  If another interval $I(b)$ has a veto mark or right endpoint at $x_i$, then $I(a)$ and $I(b)$ are not adjacent in $S$ and in $S'$.  An example of these intervals is shown on the left in Figure~\ref{distinct-fig}.  Note that no other interval can share marked points with $a_l$ or $a_v$, by our choice of $x_i$ as the smallest $x$-value shared by multiple marked points in $S$.

Now suppose there are no intervals with right endpoint at $x_i$ and there is a veto interval with veto mark $a_v$ at $x_i$.  Again we move $I(a)$ to the right by $\varepsilon$ to make $S'$.  If another interval $I(b)$ has a left endpoint or veto mark at $x_i$, then $I(a)$ and $I(b)$ are not adjacent in $S$ and in $S'$.   If another interval $I(b)$ has a left endpoint $b_l=a_r$, then $I(a)$ and $I(b)$ are adjacent in $S$ and in $S'$.  If another interval $I(b)$ has a veto mark $b_v=a_r$, then $I(a)$ and $I(b)$ are not adjacent in $S$ and in $S'$.

Now suppose there are no intervals with right endpoint or veto mark at $x_i$ and there is a veto interval with left endpoint $a_l$ at $x_i$.  In this case we move $I(a)$ and all intervals $I(c)$ for which $c_r=a_v$ and $c_v<a_l$ to the right by $\varepsilon$.  If another interval $I(b)$ has a left endpoint at $x_i$, then $I(a)$ and $I(b)$ are not adjacent in $S$ and in $S'$.  Similarly in the cases $b_l=a_v$, $b_v=a_v$, $b_r=a_v$, $b_l=a_r$, $b_v=a_r$, and $b_r=a_r$, $I(a)$ and $I(b)$ are either adjacent in both $S$ and in $S'$ or not adjacent in both $S$ and $S'$.  Examples of these intervals are shown on the right in Figure~\ref{distinct-fig}.

We iterate this process at each shared marked point, sweeping from left to right in the representation, until all marked points are distinct.  This is the representation $T$.  Since the lengths of all intervals in $T$ are the same as they were in $S$, $S$ is a unit VI-representation if and only if $T$ is.  Similarly, veto marks are at the midpoint of intervals in $T$ if and only if they are in $S$, and since there are no pairs of marked points $p_1$ and $p_2$ for which $p_1<p_2$ in $S$ and $p_2<p_1$ in $T$, $S$ is a proper VI-representation if and only if $T$ is.

\begin{figure}[ht]
\begin{center}
\includegraphics[width=12cm]{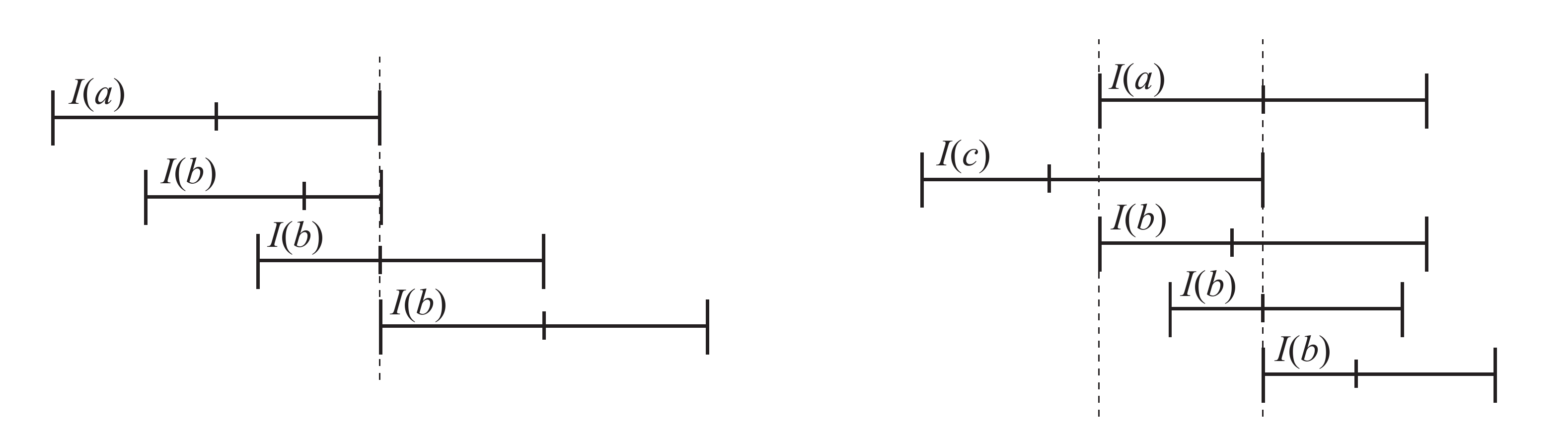}
\end{center} \caption{Some cases of Lemma~\ref{distinct_points}.} \label{distinct-fig}
\end{figure}

\end{proof}

By Lemma~\ref{distinct_points}, in what follows we may assume that every veto interval representation has distinct marked points.  Note also that given a veto interval representation $R$ with distinct marked points, its corresponding veto interval graph is completely determined by the ordering of the $3n$ marked points in $R$.

\section{Families of Veto Interval Graphs}

\begin{proposition}\label{VI families} \label{complete bipartite}
The following families of graphs are veto interval graphs.
\begin{enumerate}

\item Complete bipartite graphs $K_{m,n}$, for $m \geq 1$ and $n \geq 1$.

\item Cycles $C_n$, for $n\geq 4$.

\item Trees.

\end{enumerate}
\end{proposition}

\begin{proof}
\begin{enumerate}
\item The set of intervals with $m$ copies of the veto interval $(0,2,4)$ and $n$ copies of the interval $(3,5,7)$ is a veto interval representation of $K_{m,n}$.

\item Theorem~\ref{MUVI cycle} shows that all cycles $C_n$ with $n>3$ are MUVI graphs, and hence are veto interval graphs.

\item We will prove this by induction.
Let $T$ be a tree with $k$ vertices, $k >1$. Let $a$ be a leaf of $T$ and $v$ be its neighbor.  By induction, $T-a$ has a veto interval representation $R$.  By Lemma~\ref{distinct_points} we may assume that all marked points in $R$ are distinct.  We will add $I(a)$ to $R$ adjacent to $I(v)$ to obtain a veto interval representation of $T$. Let $I(x)$ be the interval with marked point $x_m$ closest to $v_r$ on the left, and $I(w)$ be the interval with marked point $w_m$ closest to $v_r$ on the right.  We add in $I(a)$ such that $x_m<a_l<v_r<a_v<a_r<w_m$, as shown in Figure~\ref{VI-trees}.  Note $I(a)$ and $I(v)$ are adjacent. All other intervals do not have a marked point contained in $I(a)$, so are either not overlapping with $I(a)$ or contain $I(a)$.  In either case they are not adjacent to $I(a)$, the latter by Lemma~\ref{contained}.  So $a$ is adjacent to $v$ and no other vertex in $T$, and we have obtained a veto interval representation of $T$.

\begin{figure}[ht]
\begin{center}
\includegraphics[width=6cm]{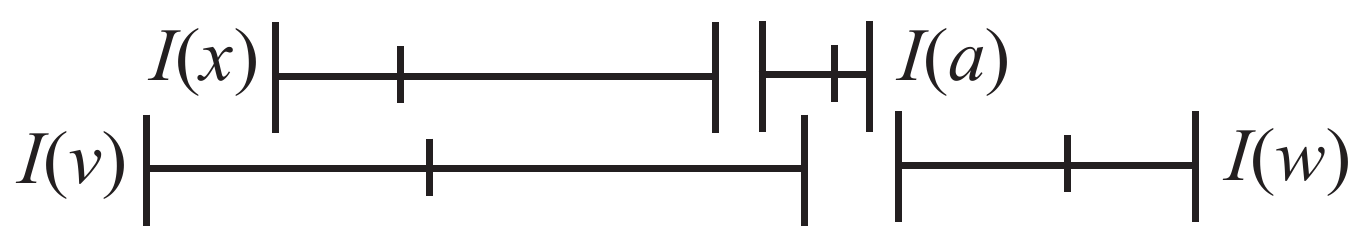}
\end{center} \caption{The induction step in part 3 of Proposition~\ref{VI families}.} \label{VI-trees}
\end{figure}
\end{enumerate}
\end{proof}

Recall that Lekkerkerker and Boland proved that both large induced cycles and asteroidal triples, some of which are trees, are forbidden induced subgraphs of interval graphs \cite{Bollek}.  However, by Proposition \ref{VI families} the same cannot be said of veto interval graphs.

We now construct a family of graphs $G_k$ for positive integers $k$, and prove that $G_{10}$ is not a VI graph.  Note that $G_{10}$ is bipartite, and hence a triangle-free graph which is not a VI graph.

Let $G_k$ be the graph with vertex set $A=\{ a_1, a_2, \ldots, a_k\}$, a vertex $v$ adjacent to every vertex in $A$, and an additional set of vertices $B=\{b_{ij} \,|\, 1 \leq i < j \leq k \}$, with $b_{ij}$ adjacent to the two vertices $a_i$ and $a_j$ in $A$.  Note that $G_k$ is a bipartite graph with $1+k+\binom{k}{2}$ vertices.
In any veto interval representation of $G_k$, veto intervals corresponding to the vertices in $A$ intersect $I(v)$ on the left or on the right.  In what follows, suppose veto intervals $I({a_1})$, $\ldots$, $I({a_p})$ intersect $I(v)$ on the left and veto intervals $I({a_{p+1}})$, $\dots$ , $I({a_{p+q}})$ intersect $v$ on the right, and let $A_1 = \{ {a_1}$, $\ldots$, ${a_p} \}$ and $A_2 = A - A_1$. So $p+q=k$.   Additionally, let $B_1=\{b_{ij} \,|\, i,j\leq p\}$, $B_2=\{b_{ij} \,|\, i,j \geq p+1\}$, and $B_3=\{b_{ij} \,|\, i \leq p \text{ and } j \geq p+1\}$.  We note that $B_1$, $B_2$, and $B_3$ partition $B$.  By construction $|B_1|=\binom{p}{2}$ and $|B_2|= \binom{q}{2}$, so $|B_3| = \binom{p+q}{2} - \binom{p}{2}-\binom{q}{2}= pq$.

\begin{lemma} \label{n>6}
There is no veto interval representation of $G_k$ with $p > 7$ or $q > 7$.
\end{lemma}

\begin{proof}
By way of contradiction, consider a veto interval representation of $G_k$ with $p>7$.  In this representation, note that an interval $I({b_{i,j}}) \in B_1$ cannot intersect $I({a_i})$ on the left and $I({a_j})$ on the right, since $I({a_i})$ and $I({a_j})$ intersect.  We may therefore further partition $B_1$ into subsets $\beta_1 = \{ b_{i,j} \,|\, I({b_{i,j}}) \text{ intersects }$ $ I({a_i}) \text{ and } I({a_j}) \text{ on the left}\}$, and $\beta_2 = \{ b_{i,j} \,|\, I({b_{i,j}}) \text{ intersects } I({a_i}) \text{ and } I({a_j})$ $ \text{ on the right}\}$. We consider the $2p-1$ regions of the line between the $2p$ left endpoints and veto marks of intervals in $A_1$, and suppose a right endpoint ${b_{i,j}}_r$ of an element $b_{i,j} \in B_1$ is in one of these regions, as shown in Figure~\ref{left-veto-regions}.  Since $b_{i,j}$ intersects intervals in $A_1$ only on the left, the set $S$ of left endpoints and veto marks of intervals in $A_1$ to the left of ${b_{i,j}}_r$ determines which vertices in $A_1$ that $b_{i,j}$ may be adjacent to, namely elements in $A_1$ with left endpoints in $S$ and veto marks not in $S$.  Since $b_{i,j}$ is adjacent to exactly two of these vertices in $G_k$, ${b_{i,j}}_v$ must veto all but two of these edges, and these two must be the vertices $a_i$ and $a_j$ in $S$ with left endpoints furthest to the right.

\begin{figure}[ht!]
\begin{center}
\includegraphics[width=8cm]{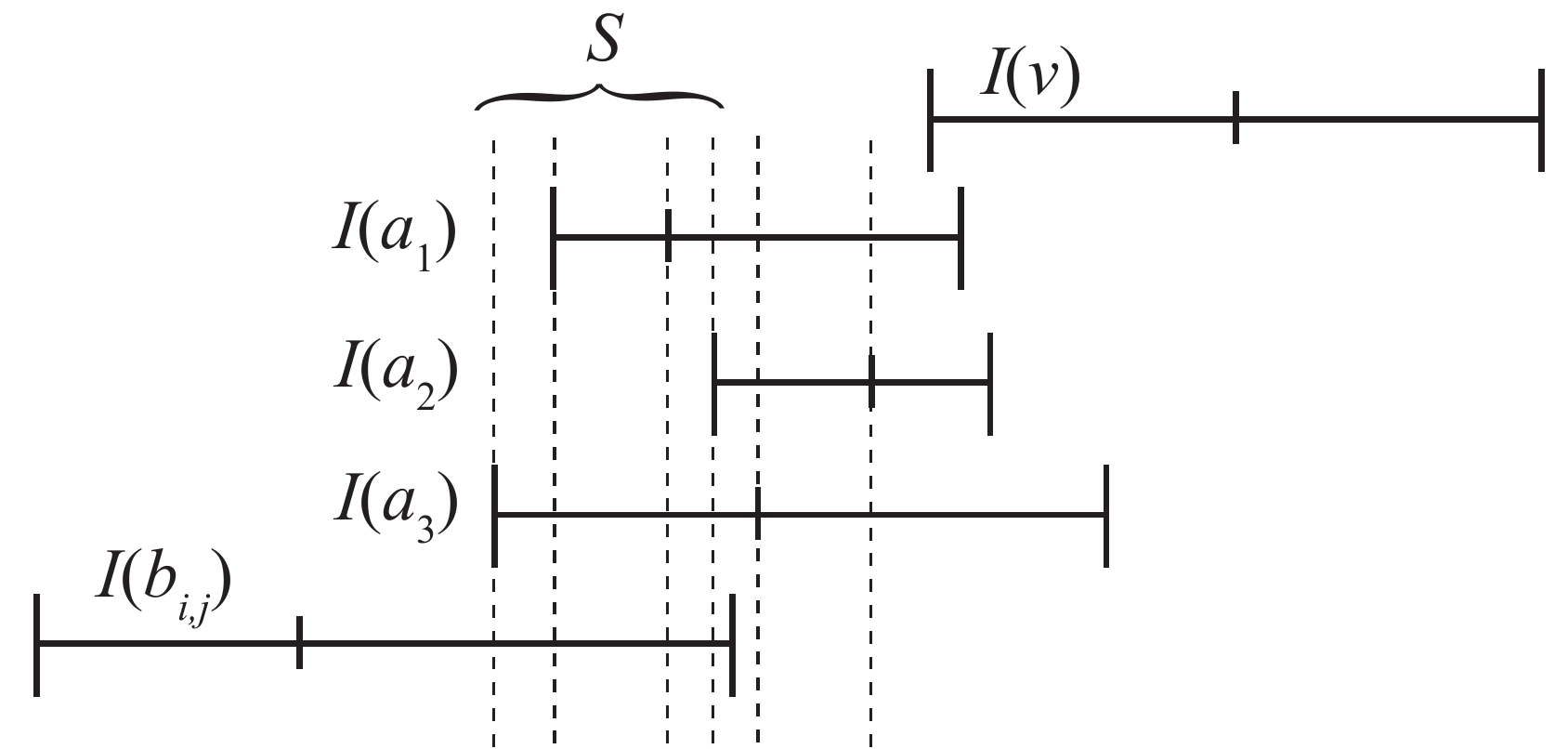}
\end{center} \caption{Regions determined by left endpoints and veto marks of intervals in $A_1$ in Lemma~\ref{n>6}.} \label{left-veto-regions}
\end{figure}

Therefore if two vertices $b_{i,j}$ and $b_{s,t}$  in $\beta_1$ have their right endpoints in the same region, they are adjacent to the same two vertices, i.e. $i=s$ and $j=t$.  So there is at most one right endpoint ${b_{i,j}}_r$ in each of these $2p-1$ regions.  Additionally, there is no such right endpoint in the region furthest to the left, since in this case $I({b_{i,j}})$ intersects only one interval $I({a_i})$, and there is no such right endpoint in the region farthest to the right, since in this case $I({b_{i,j}})$ intersects the veto mark of every interval $I({a_i})$.

Therefore $|\beta_1| \leq 2p-3$.  Using an equivalent argument, $|\beta_2| \leq 2p-3$, so $|B_1| \leq 4p-6$.  Since $|B_1|=\binom{p}{2}$ we have $\binom{p}{2} \leq 4p-6$, which yields $p \leq 7$.  Equivalently, $q \leq 7$.

\end{proof}

\begin{lemma} \label{2k-4}
There is no veto interval representation of $G_k$ in which $pq >2k-4$.
\end{lemma}

\begin{proof}
  We prove that in any veto interval representation of $G_k$, $|B_3| \leq 2k-4$.

Let $R$ be a veto interval representation of $G_k$ with distinct endpoints.  Let $a \in A_1$ such that $a_r \geq x_r$ for all $x \in A_1$, and $b \in A_1$ such that $b_v \leq x_v$ for all $x \in A_1$.  Similarly, let $c \in A_2$ such that $c_l \leq x_l$ for all $x \in A_2,$ and $d \in A_2$ such that $d_v \geq x_v$ for all $x \in A_2.$  We will prove that for every $b_{ij} \in B_3$, either $a_i = a$, $a_i = b$, $a_j = c$ or $a_j =d$.

Assume not for contradiction, and consider $I(b_{i,j})$.  Since $b_{{i,j}_l} < {a_i}_r < {a_j}_l < b_{{i,j}_r}$, $I(b_{{i,j}})$ intersects both $I(a)$ and $I(c)$.  By assumption, $a_i \neq a$ and $a_j \neq c$, so either $b_{i,j}$ vetos $a$ or $a$ vetos $b_{i,j}$, and similarly for $c$.  Assume without loss of generality that $v_v < b_{{i,j}_v}$.  Therefore $b_{i,j}$ cannot veto $a$, so $b_{{i,j}_l} < a_v$.  But since $a_v < v_l < a_{t_r}$ for all $1 \leq t \leq p$, $b_{i,j}$ overlaps $a_t$ for all $1 \leq t \leq p$.  Therefore $a_t$ vetos $b_{i,j}$ for all $1 \leq t \leq p$, $t \neq i$.  This contradicts the fact that $b_v < x_v$ for all $x \in A_1 - \{b\}$.

Note that for every $i, 1 \leq i \leq p$ there are exactly two possible values of $j$ for $b_{i,j}$.  Similarly, there for every $j, p+1 \leq j \leq p +q$ there are exactly two possible values of $i$ for $b_{i,j}$.  Thus gives us $|B_3| \leq 2p + 2q - 4 = 2k -4$ with subtracting out the overcounting.

Combining the facts that $|B_3| = pq$ and $|B_3| \leq 2k-4$, we can see that there is no veto interval representation of $G_k$ in which $pq >2k-4$.

\end{proof}

\begin{theorem} \label{G9}
$G_{10}$ is not a VI graph.
\end{theorem}

\begin{proof}
Suppose $G_{10}$ has a veto interval representation, with $p$ and $q$ defined above, and $p \leq q$.  By Lemma~\ref{n>6}, we have $p=3$, $p=4$, or $p=5$.  In each case we have $pq>16$, contradicting Lemma~\ref{2k-4}.
\end{proof}

$G_{10}$ is the smallest bipartite graph we have found which is not a veto interval graph, with $56$ vertices.

\begin{corollary}
Not every subgraph of a VI graph is a VI graph.
\end{corollary}

\begin{proof}
The graph $G_{10}$ is not a veto interval graph, and is bipartite, but all complete bipartite graphs are veto interval graphs by Lemma~\ref{complete bipartite}.
\end{proof}

\begin{lemma}
Up to isomorphism, the directed graphs shown in Figure~\ref{5-cycle-orientations} are the only orientations of $C_5$ which are directed veto interval graphs. \label{5-cycles}
\end{lemma}

\begin{figure}[ht]
\begin{center}
\includegraphics[width=0.75 \linewidth]{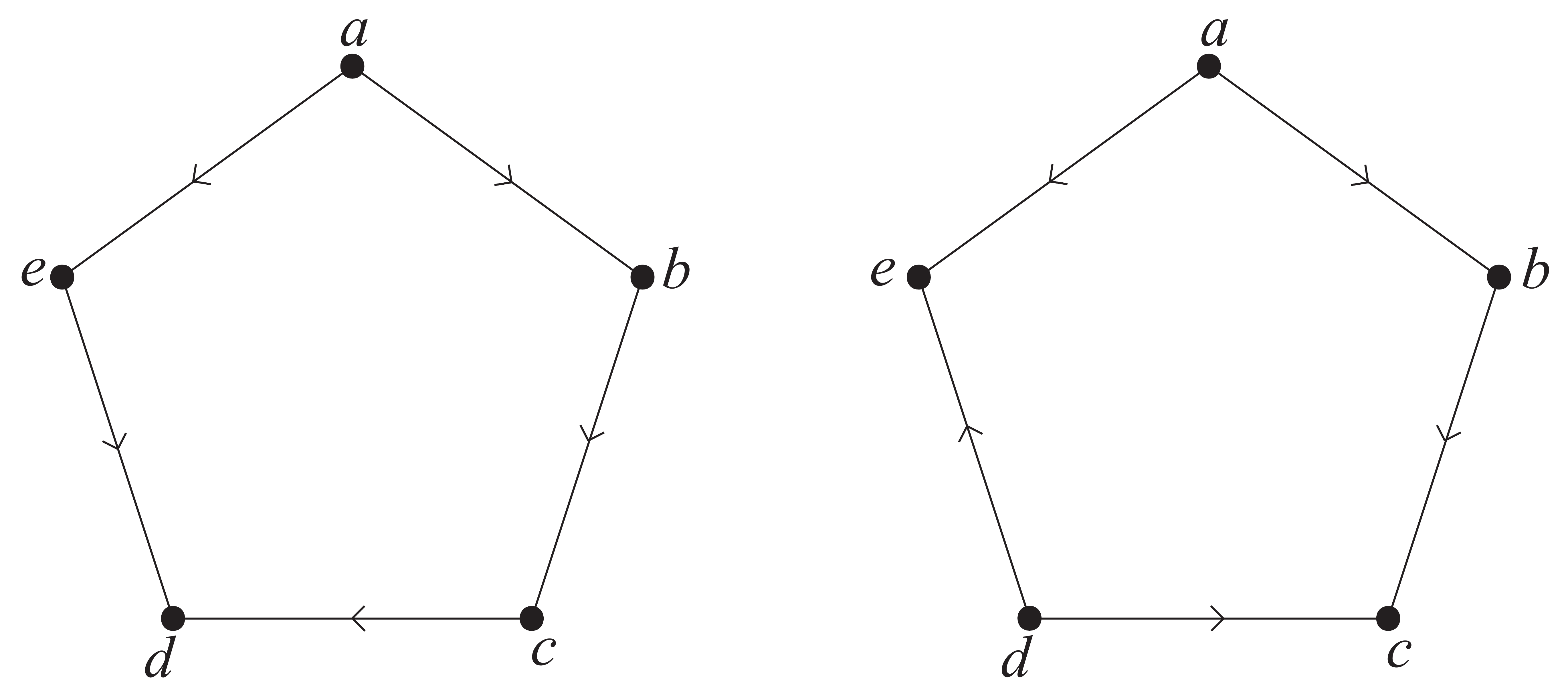}
\end{center} \caption{The two possible orientations of $C_5$ as a directed VI graph.} \label{5-cycle-orientations}
\end{figure}

\begin{figure}[ht]
\begin{center}
\includegraphics[width=3.5in]{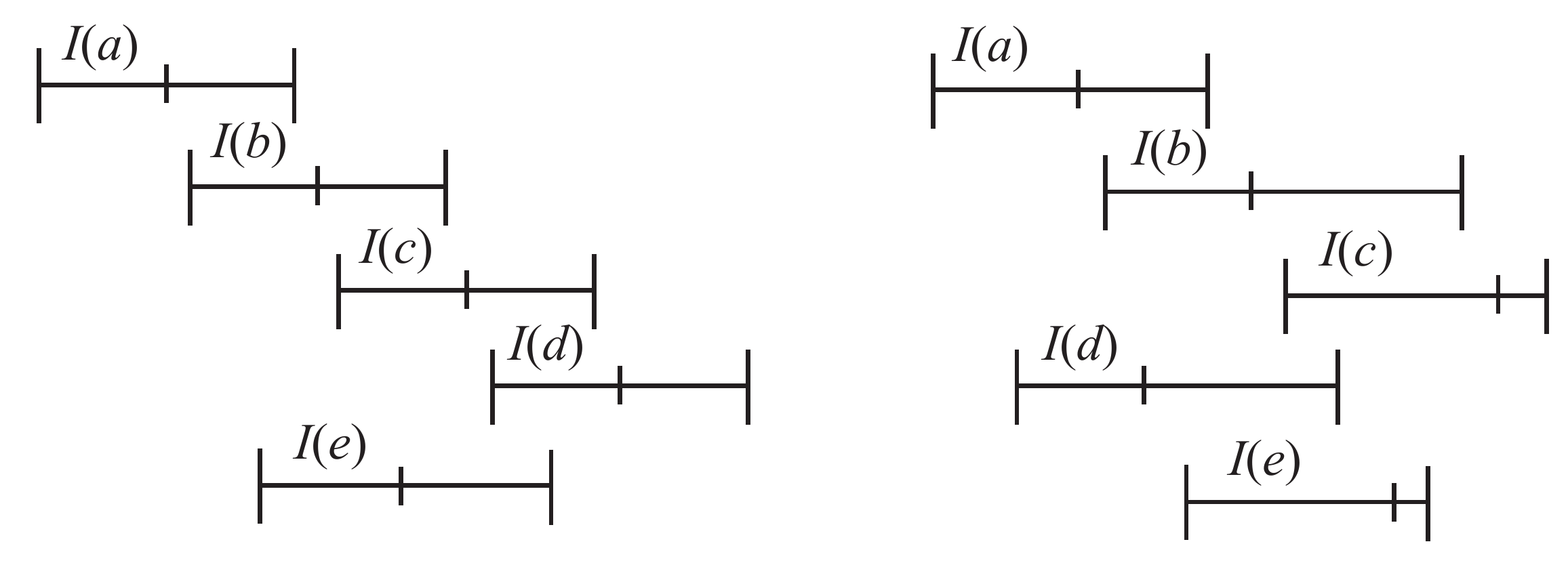}
\end{center} \caption{VI representations of the two orientations of $C_5$ in Figure~\ref{5-cycle-orientations}.} \label{5-cycle-reps}
\end{figure}

\begin{proof}
Let $G$ be a directed veto interval graph which is an orientation of $C_5$, with vertices $a$, $b$, $c$, $d$, and $e$.  The graph $G$ does not have a directed path of length 4 by Lemma~\ref{directed-cycles}.  But $G$ must have a directed path of at least 2 edges, since otherwise $G$ would have alternating directed edges, which can only happen for even cycles.  So the longest directed path in $G$ has exactly 2 or 3 edges.

If $G$ has longest directed path $a \to b \to c \to d$, then $G$ also has edges $a \to e$ and $e \to d$ or this directed path would be longer.  This yields the orientation shown on the left in Figure~\ref{5-cycle-orientations}.  If $G$ has longest directed path $a \to b \to c$, then $G$ also has directed edges $a \to e$ and $d \to c$.  Then $a \to e \to d \to c$ is a longer path unless $G$ also has edge $d \to e$.  This yields the orientation shown on the right in Figure~\ref{5-cycle-orientations}.  Each of these orientations is a directed veto interval graph as shown in Figure~\ref{5-cycle-reps}.
\end{proof}

\begin{theorem}
The Gr\"{o}tzsch Graph is not a VI graph. \label{grotzsch-not-VI}
\end{theorem}

\begin{figure}[ht]
\begin{center}
\includegraphics[width=50mm]{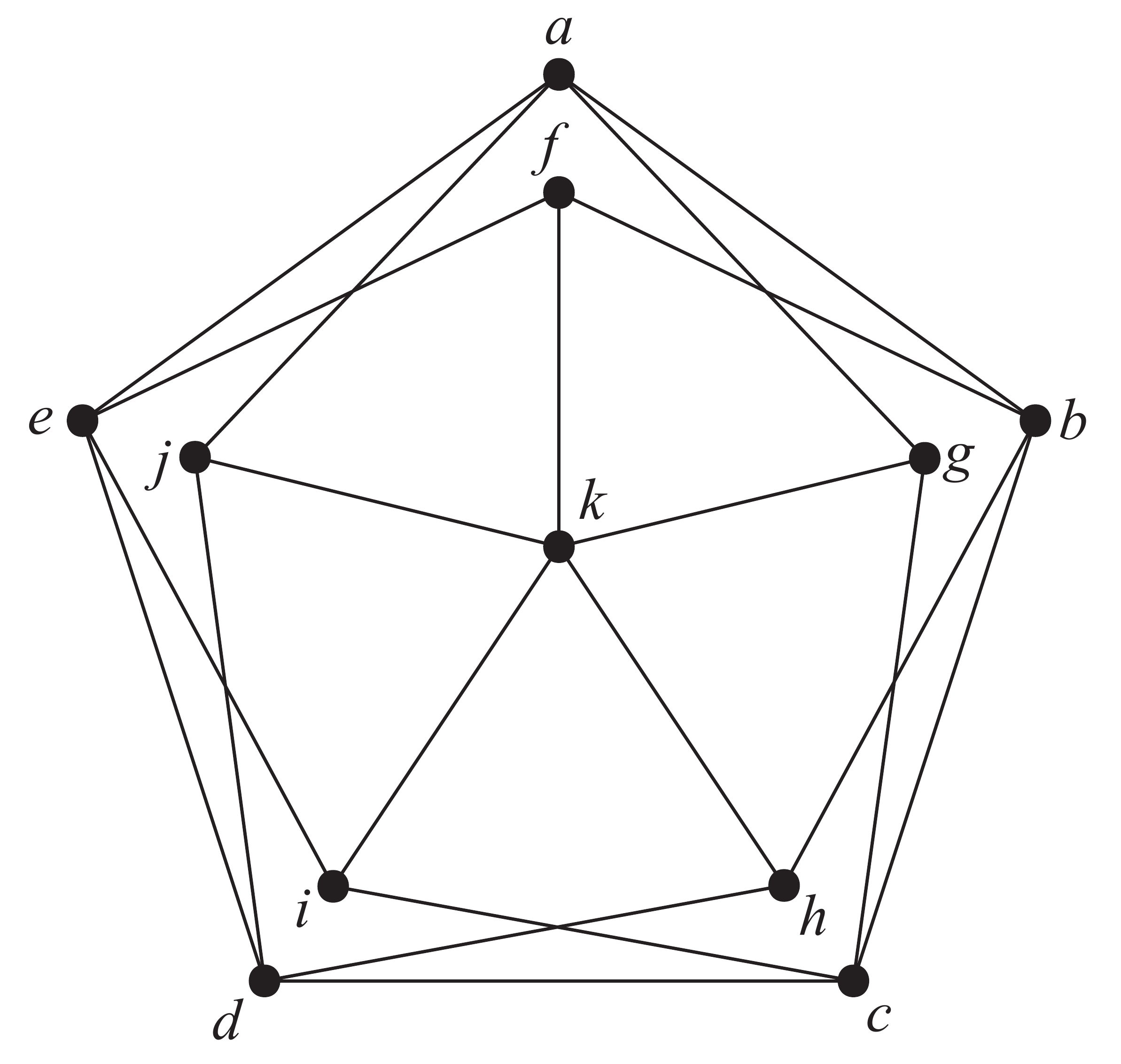}
\end{center} \caption{The Gr\"{o}tzsch graph.} \label{grotzsch}
\end{figure}

\begin{proof}
Label the Gr\"{o}tzsch graph $G$ as shown in Figure~\ref{grotzsch}. Assume by way of contradiction that the Gr\"{o}tzsch Graph has a VI representation $R$.  Direct the edges of $G$ using $R$ to obtain $G$ as a directed VI graph.  We will show this orientation of $G$ fails to satisfy Lemma~\ref{directed-cycles}.

By Lemma~\ref{5-cycles}, the $5$-cycle $abcde$ must be directed in one of the two ways shown in Figure~\ref{5-cycle-orientations}.

\noindent \textbf{Case 1. $G$ has directed edges $a \to b \to c \to d$ and $a \to e \to d$.} Then the additional orientations on the edges $a \to g$, $g \to c$, $b \to h$, $h \to d$, $a \to j$, and $j \to d$, shown in the first drawing in Figure~\ref{Grotzsch-case1}, are forced by Lemma~\ref{directed-cycles}.
\bigskip

\noindent \textbf{Subcase 1a. $G$ has directed edge $j \to k$.} Then the directed edges $h \to k$, $g \to k$, $b \to f$, $f \to k$, $e \to f$, $e \to i$, and $i \to k$, shown in the second drawing in Figure~\ref{Grotzsch-case1}, are forced again by Lemma~\ref{directed-cycles}.  Applying Lemma~\ref{directed-cycles} to the cycle $cdei$ forces $c \to i$ but applying Lemma~\ref{directed-cycles} to the cycle $gcik$ forces $i \to c$.  Hence we have reached a contradiction.
\bigskip

\noindent \textbf{Subcase 1b. $G$ has directed edge $k \to j$.} Then the directed edges $k \to g$, $k \to h$, $k \to i$, $i \to c$, $i \to e$, $k \to f$, and $f \to e$ shown in the third drawing in Figure~\ref{Grotzsch-case1}, are forced again by Lemma~\ref{directed-cycles}.  Applying Lemma~\ref{directed-cycles} to the cycle $bfhk$ forces $b \to f$ but applying Lemma~\ref{directed-cycles} to the cycle $abfe$ forces $f \to b$, which is a contradiction.

\begin{figure}[ht]
\begin{center}
\includegraphics[width=0.3\textwidth]{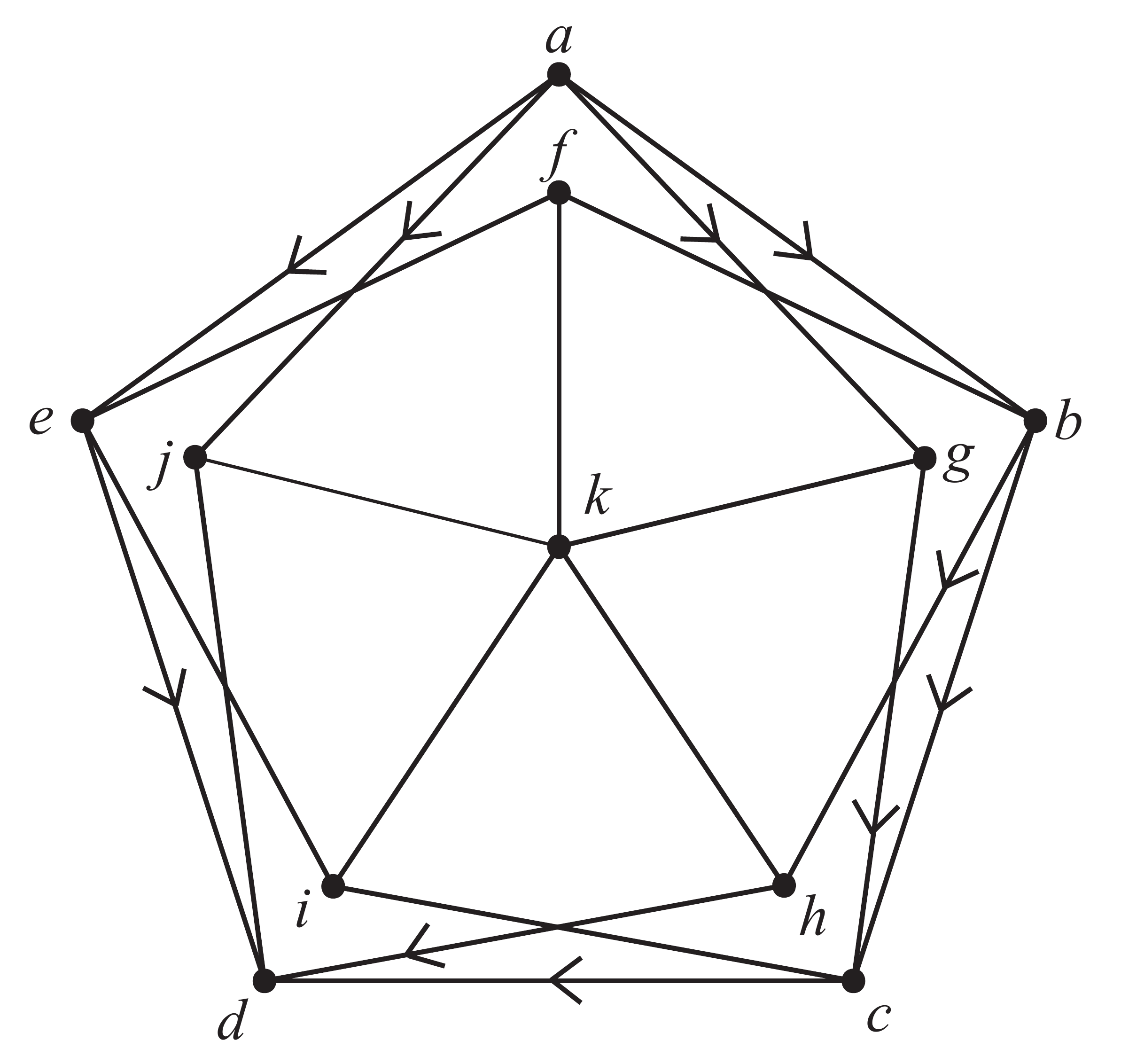} \includegraphics[width=0.3\textwidth]{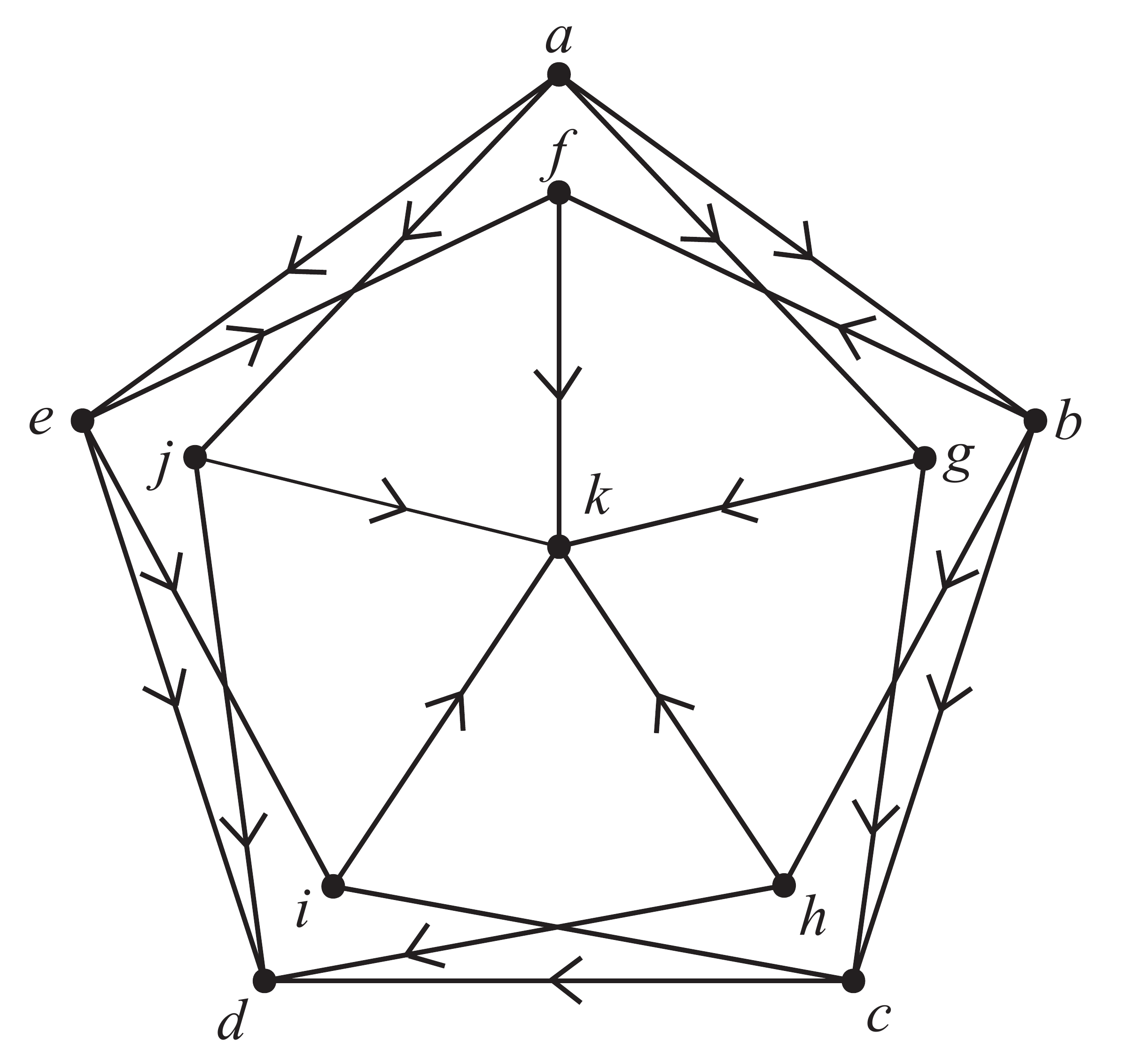} \includegraphics[width=0.3\textwidth]{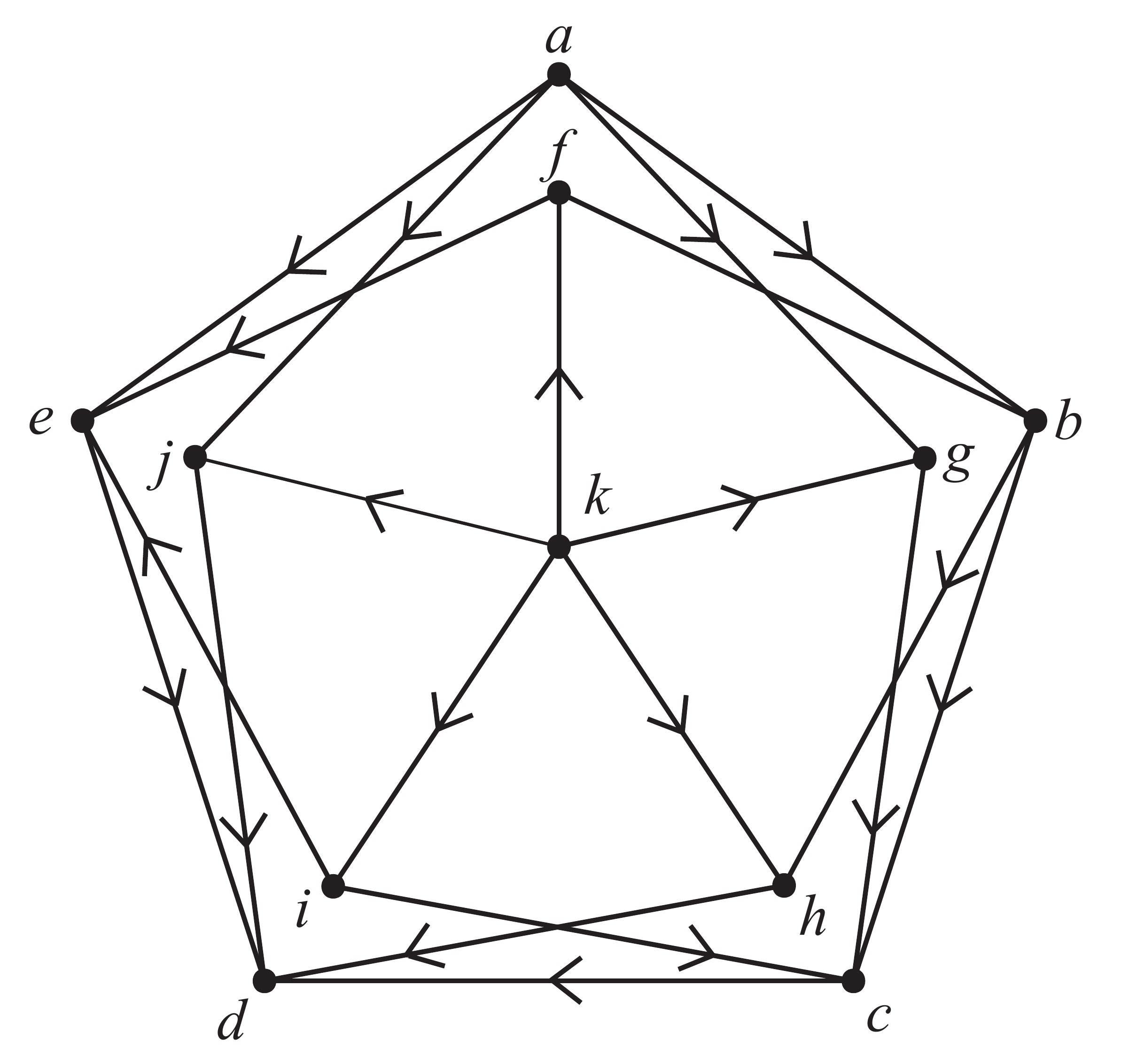}
\end{center} \caption{Orientations of the Gr\"{o}tzsch graph in Case 1 of the proof of Theorem~\ref{grotzsch-not-VI}.} \label{Grotzsch-case1}
\end{figure}

\noindent \textbf{Case 2. $G$ has directed edges $a \to b \to c$, $a \to e$, $d \to e$, and $d \to c$.} Then the additional orientations $a \to g$ and $g \to c$ are forced by Lemma~\ref{directed-cycles}, as shown in the first drawing of Figure~\ref{Grotzsch-case2}.
\bigskip

\noindent \textbf{Subcase 2a. $G$ has directed edge $g \to k$.} Then the directed edges $a \to j$, $j \to k$, $d \to j$, $d \to h$, $h \to k$, $b \to h$, $b \to f$, $f \to k$, $e \to f$, $e \to i$, and $i \to k$, shown in the second drawing in Figure~\ref{Grotzsch-case2}, are forced again by Lemma~\ref{directed-cycles}.  Applying Lemma~\ref{directed-cycles} to the cycle $cdei$ forces $c \to i$ but applying Lemma~\ref{directed-cycles} to the cycle $gcik$ forces $i \to c$, which is a contradiction.
\bigskip

\noindent \textbf{Subcase 2b. $G$ has directed edge $k \to g$.} Then the directed edges $k \to i$, $i \to c$, $i \to e$, $k \to f$, $f \to e$, $f \to b$, $k \to h$, $h \to b$, $h \to d$, $k \to j$, and $j \to d$, shown in the third drawing in Figure~\ref{Grotzsch-case2}, are forced again by Lemma~\ref{directed-cycles}.  Applying Lemma~\ref{directed-cycles} to the cycle $agkj$ forces $a \to j$ but applying Lemma~\ref{directed-cycles} to the cycle $ajde$ forces $j \to a$, which is a contradiction.

\begin{figure}[ht]
\begin{center}
\includegraphics[width=0.3\textwidth]{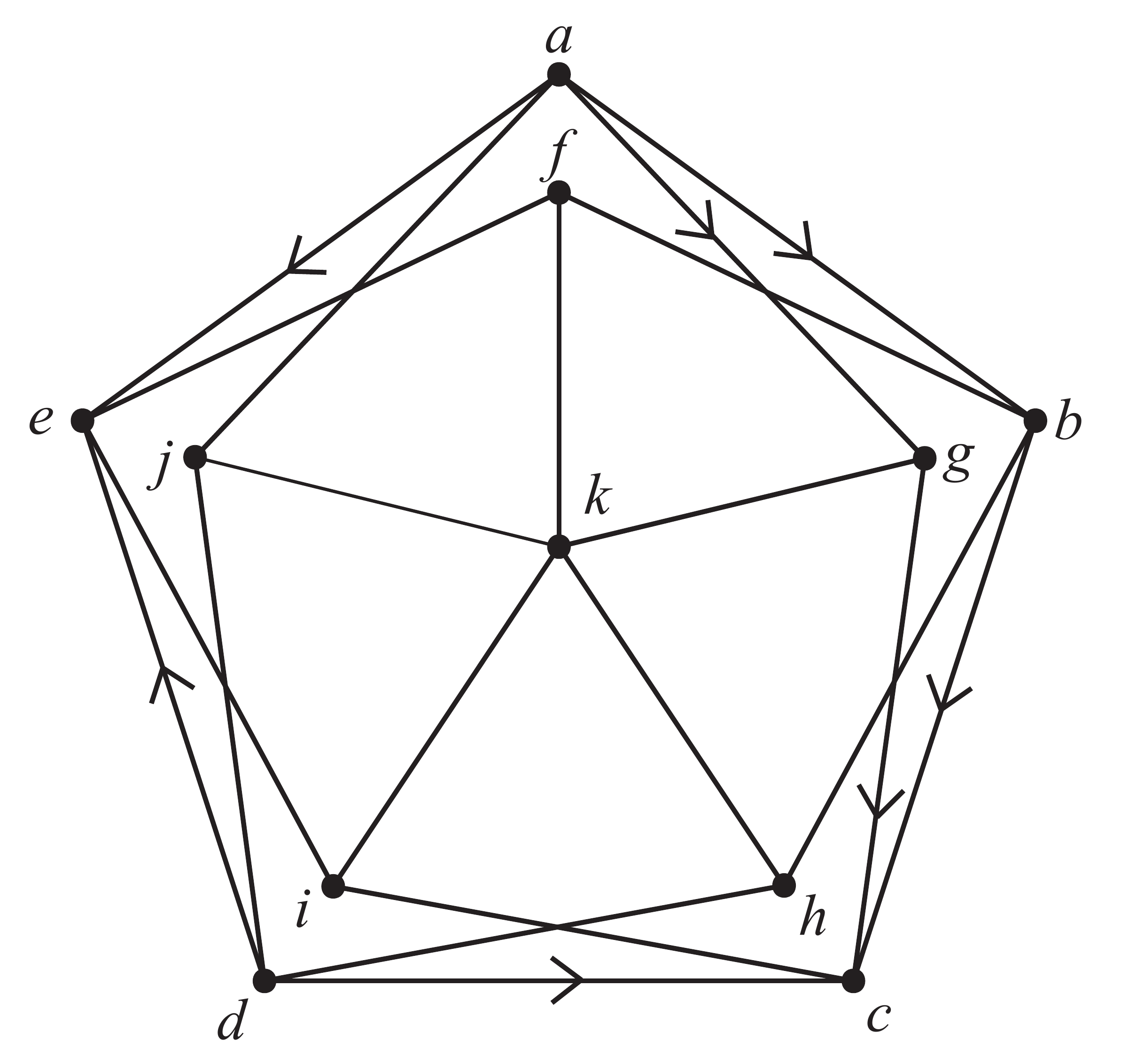} \includegraphics[width=0.3\textwidth]{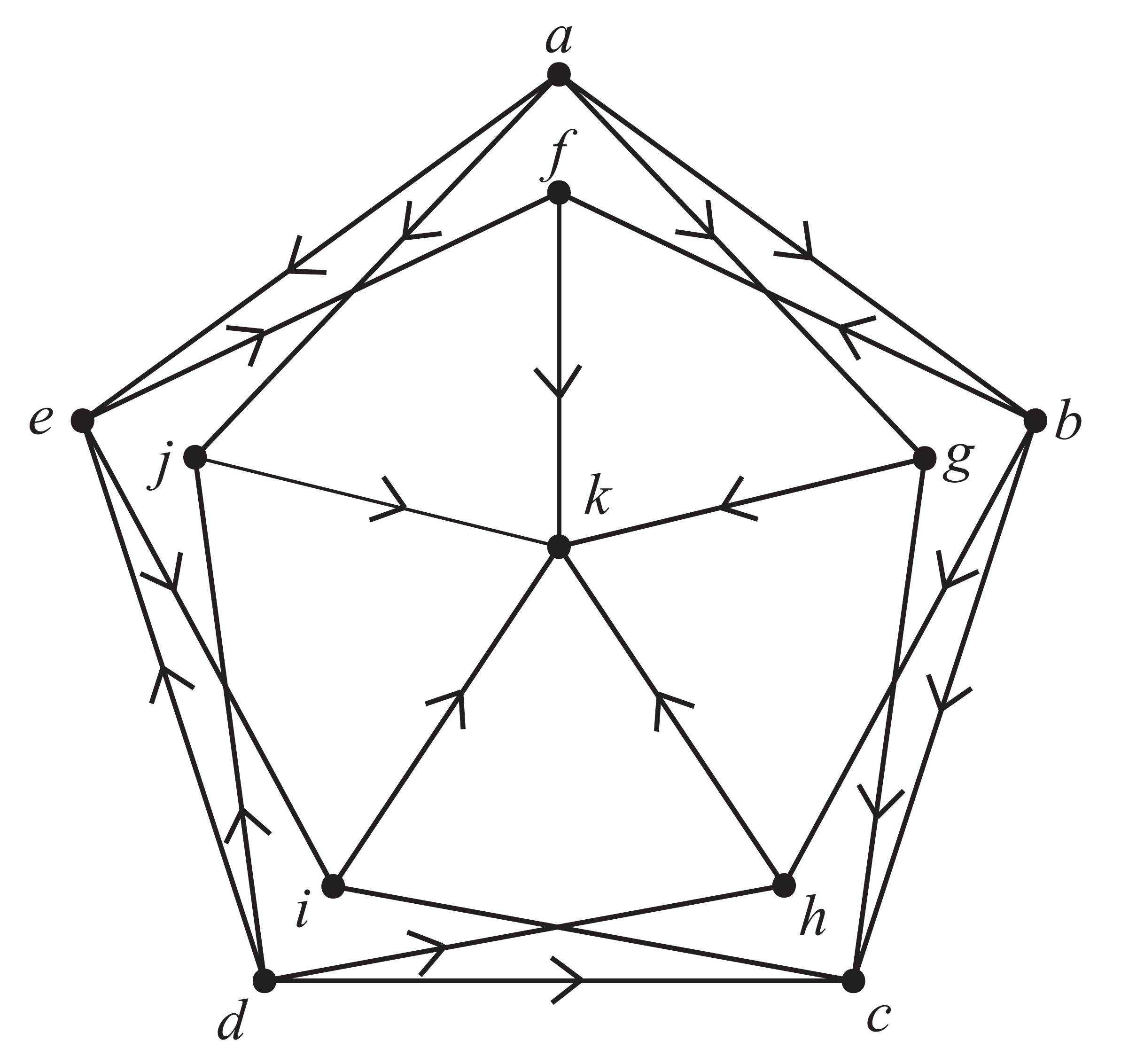} \includegraphics[width=0.3\textwidth]{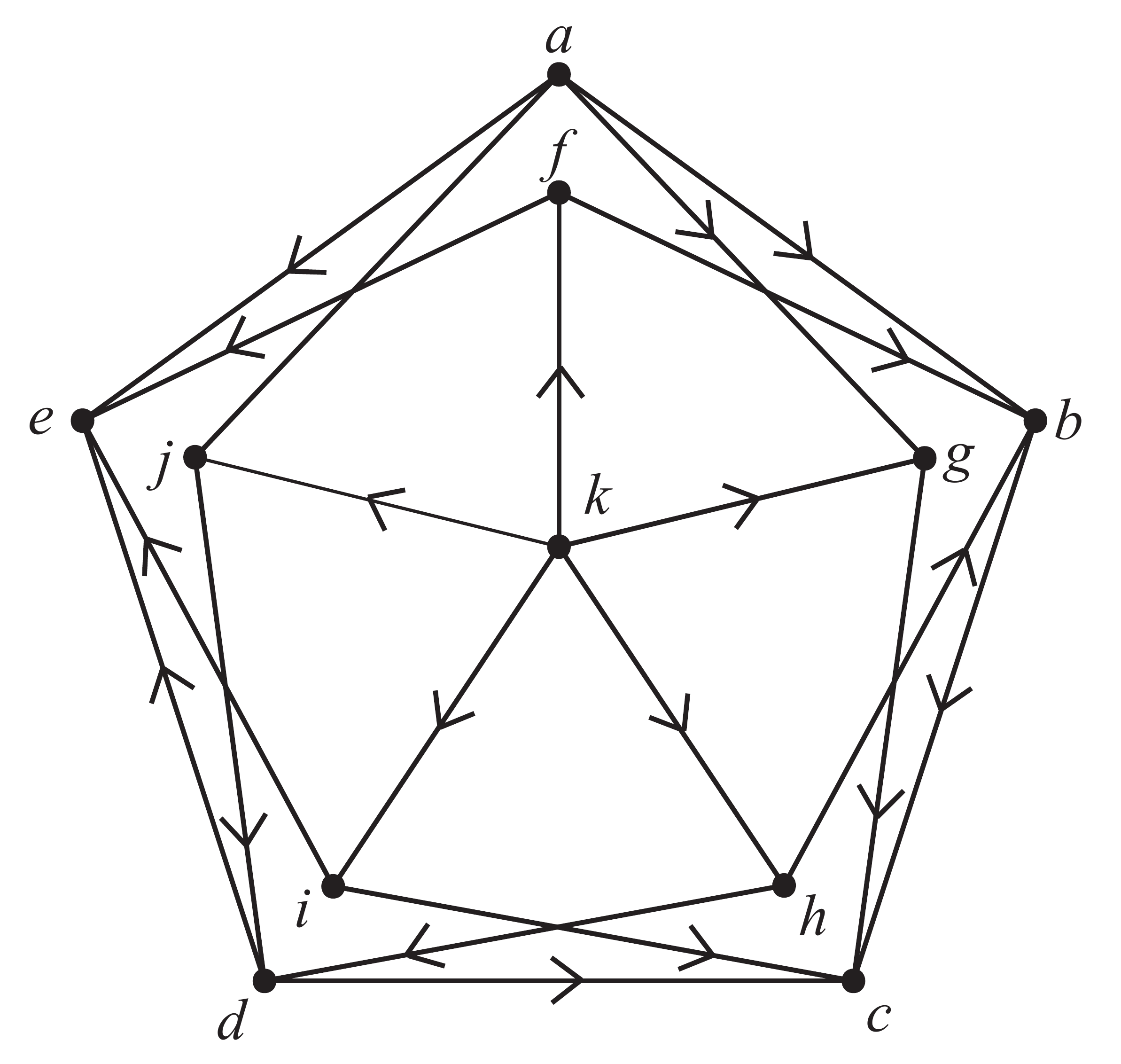}
\end{center} \caption{Orientations of the Gr\"{o}tzsch graph in Case 2 of the proof of Theorem~\ref{grotzsch-not-VI}.} \label{Grotzsch-case2}
\end{figure}

Since we have reached a contradiction in all cases, the Gr\"{o}tzsch Graph is not a VI graph.

\end{proof}

\begin{lemma}
The Gr\"{o}tzsch Graph is a minimal forbidden subgraph for VI graphs.
\end{lemma}

\begin{proof}
Let $G$ be the Gr\"{o}tzsch Graph. Due to symmetry, it is sufficent to show that $G-k$, $G-i$, and $G-a$ are all VI graphs. Their VI representations are shown in Figure~\ref{grotzsch-minimal}.

\begin{figure}[ht]
\centerline{
\includegraphics[width=0.5\textwidth]{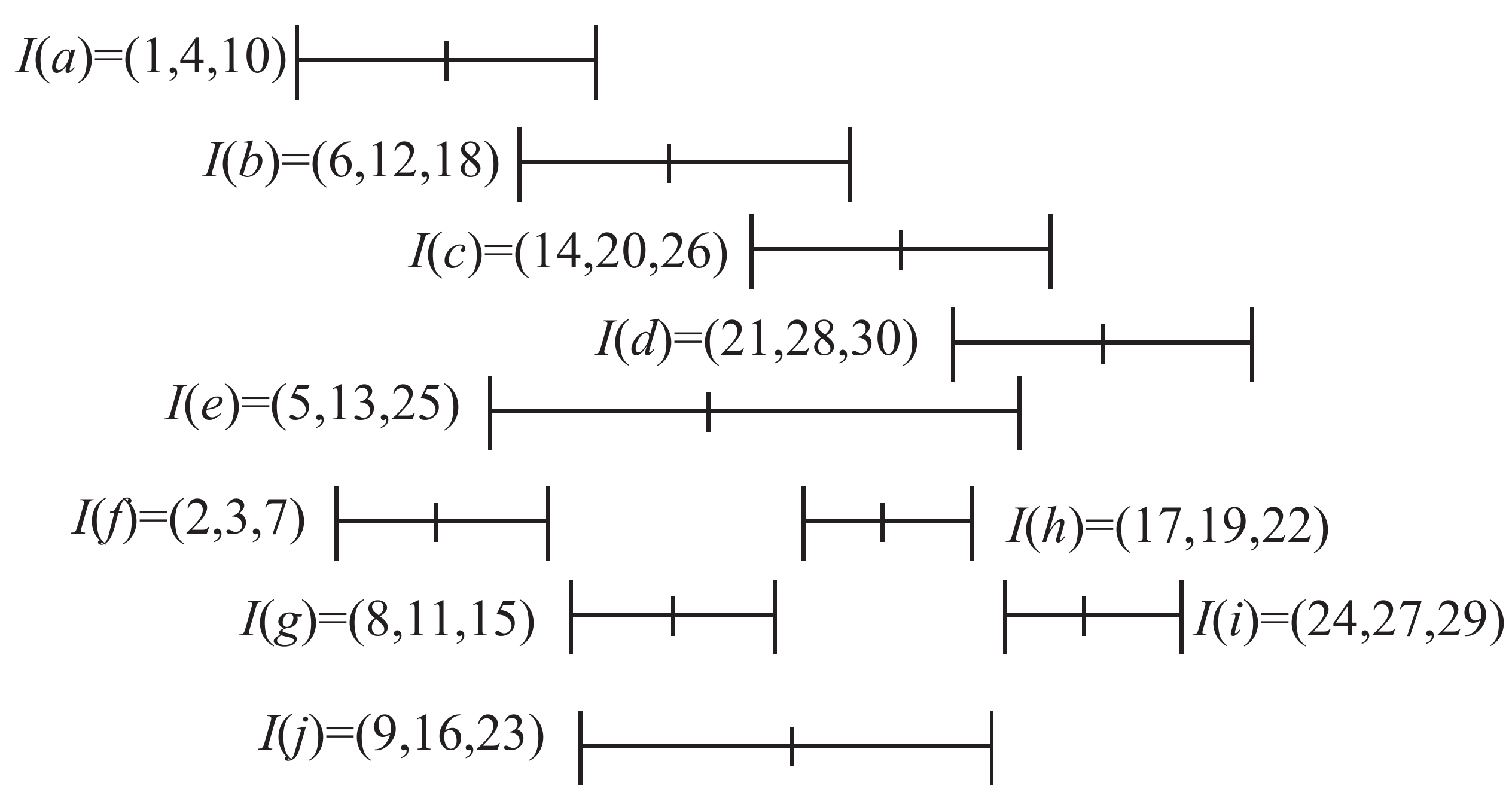} \includegraphics[width=0.5\textwidth]{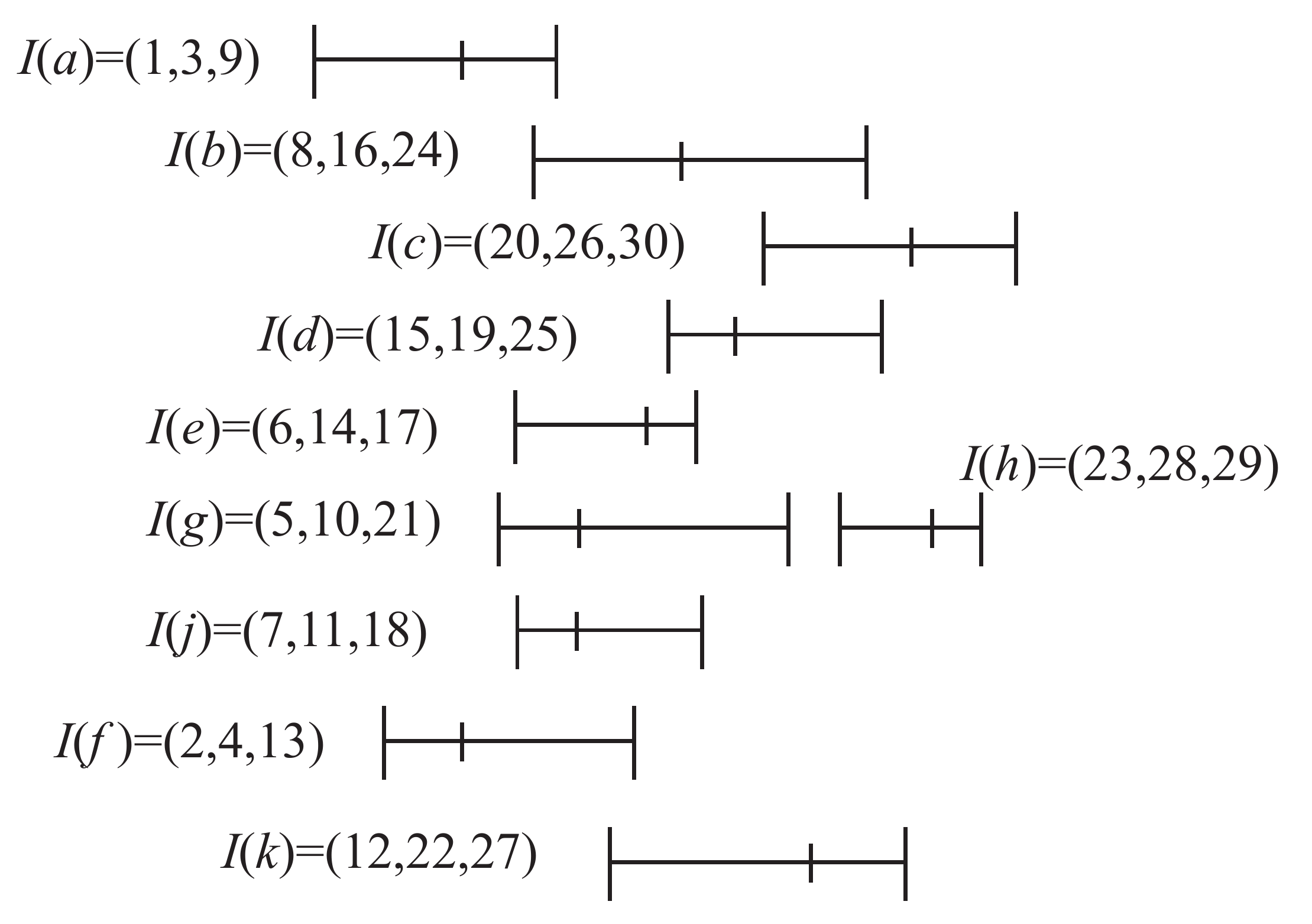} \includegraphics[width=0.5\textwidth] {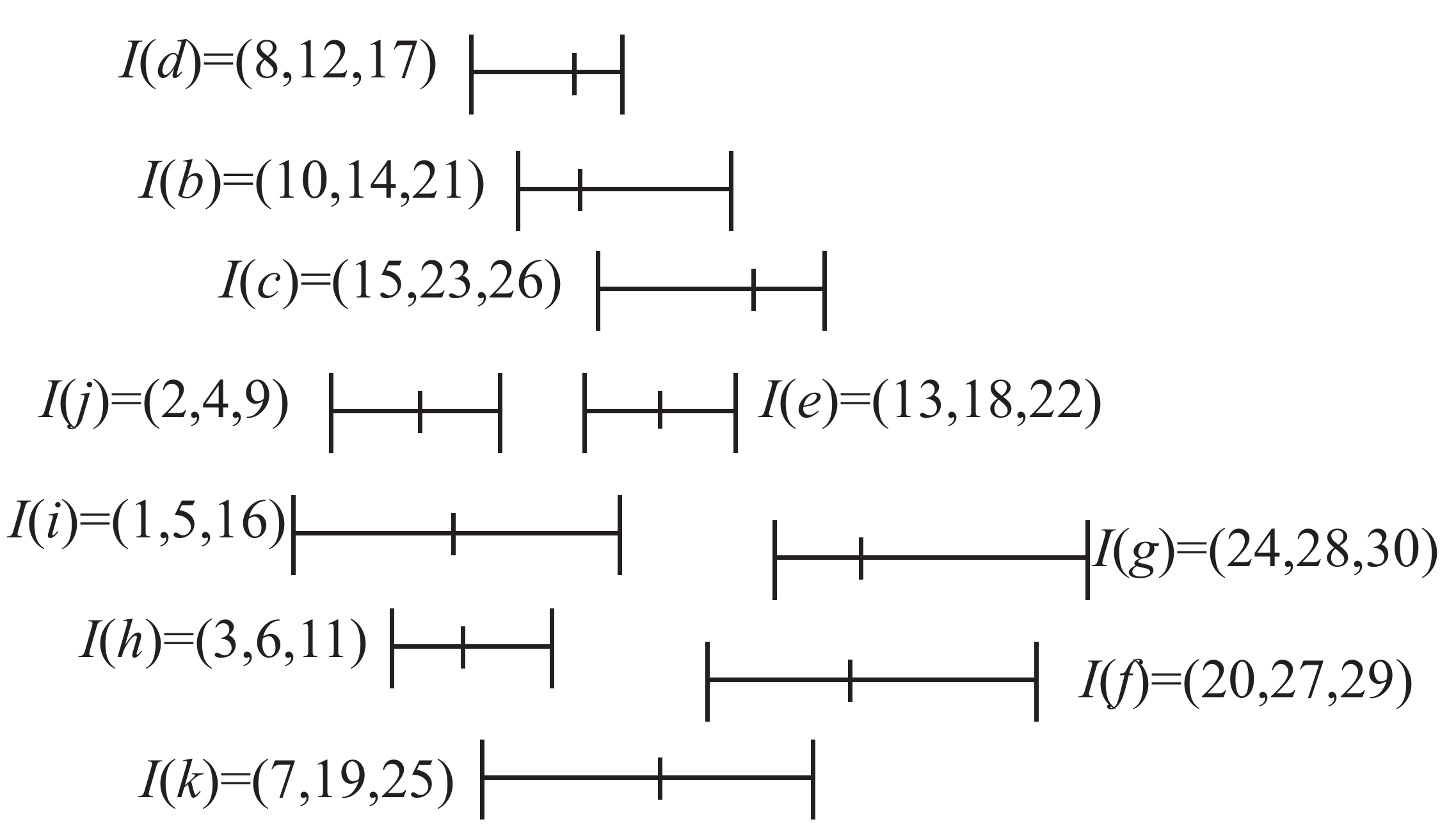}
}
\caption{VI representations of the induced subgraphs of the Gr\"{o}tzsch graph $G-k$, $G-i$, and $G-a$.} \label{grotzsch-minimal}
\end{figure}

\end{proof}

\section{Unit, proper, and midpoint VI graphs} \label{unit-midpoint-section}

In this section we focus on the unit, proper, and midpoint veto interval graphs defined in Section~\ref{introduction-section}.  First we prove theorems about MUVI graphs (the most restricted class), and then we investigate the relationships between UVI, MVI, MUVI, PVI, and MPVI  graphs.

Recall that a graph $G$ is a \textit{\textbf{caterpillar}} if $G$ is a tree with a path $s_1 s_2 \cdots s_k$, called the \textit{\textbf{spine}} of $G$, such that every vertex of $G$ has distance at most one from the spine.

\begin{theorem}\label{MUVI cycle}
The following families of graphs are MUVI graphs.
\begin{enumerate}

\item Complete bipartite graphs $K_{m,n}$, for $m \geq 1$ and $n \geq 1$.

\item Caterpillars.

\item The cycle $C_n$, for $n\geq 4$.
\end{enumerate}
\end{theorem}

\begin{proof}
\begin{enumerate}

\item  The representation given in Proposition \ref{complete bipartite} is a MUVI representation of $K_{m,n}$, with intervals of length $4$.

\item Let $G$ be a caterpillar with spine vertices $s_1, s_2, ..., s_k$ and additional vertices $a_1$, $a_2$, $\ldots$, $a_j$.  We let $I(s_1)=(1,2,3)$, $I({s_2})=(3,4,5)$, $\ldots$, $I({s_k})=(2k-1, 2k, 2k+1)$, and for every vertex $a$ adjacent to $s_i$, $I(a)=(2i-3+i/(k+1), 2i-2+i/(k+1), 2i-1+i/(k+1))$.  Figure~\ref{caterpillar} shows an example of a MUVI representation of a caterpillar using this construction.


\begin{figure}[ht]
\begin{center}
\includegraphics[width=12cm]{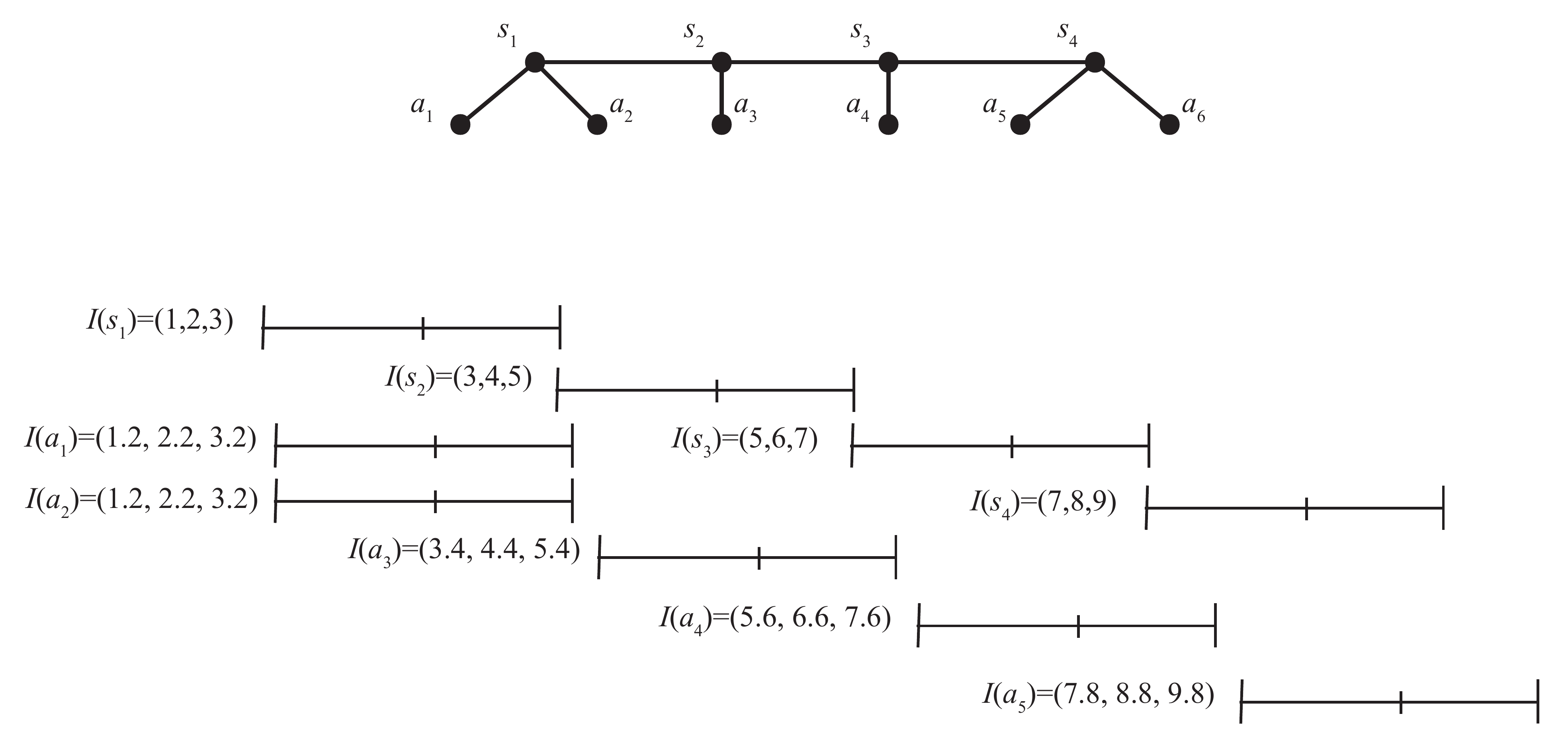}
\end{center} \caption{A caterpillar and its VI representation.} \label{caterpillar}
\end{figure}

\item Label the vertices of $C_n$ for $n \geq 4$ as $v_1, v_2, ..., v_n$.  If $n$ is odd, let $I(v_k) = (20k-18, 20k-6, 20k+6)$ for $1 \leq k \leq \frac{n+1}{2}$, $I(v_{\frac{(n+3)}{2}}) = (10n-22, 10n-10,10n+2)$, $I(v_i) = (20(n-i)+12, 20(n-i)+24, 20(n-i)+36)$ for $\frac{n+5}{2} \leq i \leq n-1$, and $I(v_n) = (15,27,39).$  If $n$ is even, let $I(v_k) = (20k-18, 20k-6, 20k+6)$ for $1 \leq k \leq \frac{n}{2}$, $I(v_{\frac{(n+2)}{2}}) = (10n-4, 10n+8,10n+20)$, $I(v_i) = (20(n-i)+12, 20(n-i)+24, 20(n-i)+36)$ for $\frac{n+4}{2} \leq i \leq n-1$, and $I(v_n) = (15,27,39).$  See Figure~\ref{MUVI-Ck} for the MUVI representation for $C_9$.

\begin{figure}[ht]
\begin{center}
\includegraphics[width=\textwidth]{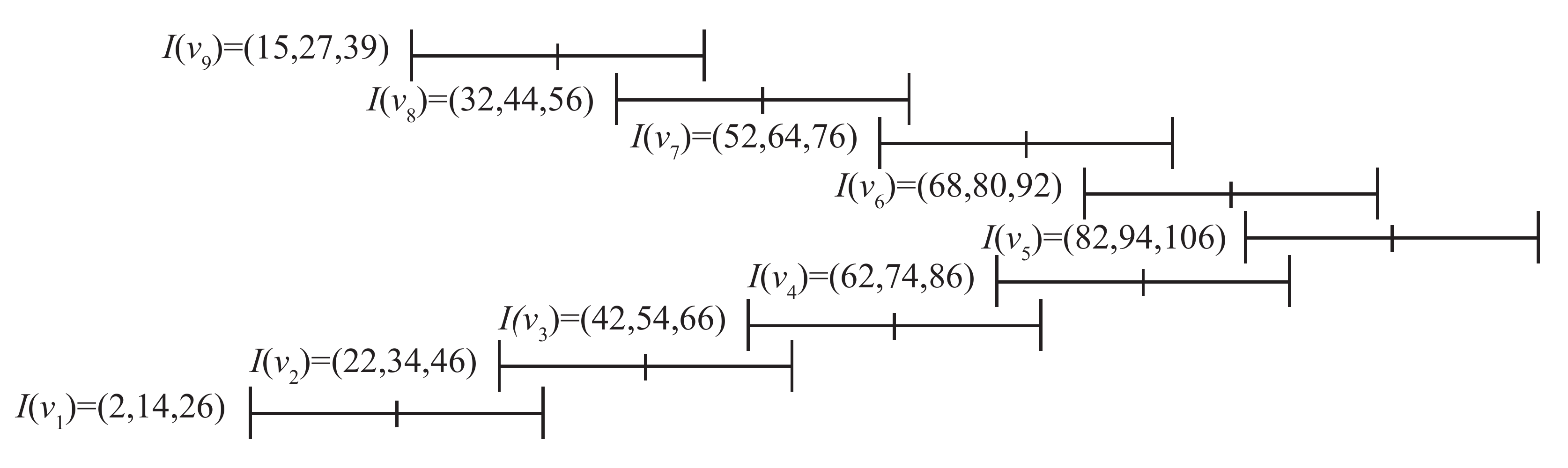}
\end{center} \caption{The midpoint unit veto interval representation of $C_9$.} \label{MUVI-Ck}
\end{figure}

\end{enumerate}
\end{proof}

The \textit{\textbf{$5$-lobster}} $L_5$, shown in Figure~\ref{5-lobster}, is the graph formed by subdividing each edge of the star $S_5$.

\begin{lemma}
The $5$-lobster $L_5$ is not a MUVI graph. \label{5-lobster-lemma}
\end{lemma}

\begin{figure}[ht]
\begin{center}
\includegraphics[width=5cm]{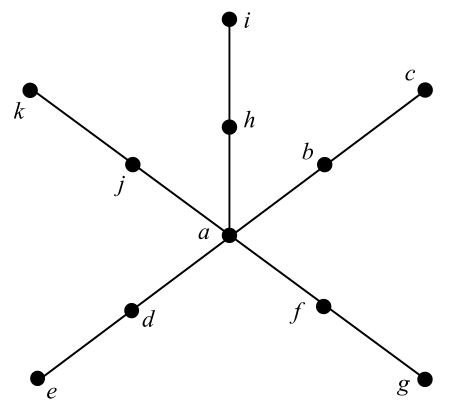}
\end{center} \caption{The 5-lobster.} \label{5-lobster}
\end{figure}

Label the vertices of $L_5$ as in Figure~\ref{5-lobster}.  Assume $R$ is a MUVI representation of $L_5$.   Since $a$ has degree 5, by the pigeonhole principle, at least three intervals intersect $I(a)$ on the same side. Without loss of generality, let $I(b)$, $I(d)$, and $I(f)$ intersect $I(a)$ on the right. Also, without loss of generality, let $b_l<d_l<f_l$. Then, $f_l<b_v<d_v<f_v<b_r<d_r<f_r$. The vertex $e$ must be adjacent to $d$ only. If $e$ intersects $d$ on the right, then $d_v<e_l$. If $e_l<b_r$, then $b_v<e_l<b_r<d_r<e_v$, so $e$ is adjacent to $b$, a contradiction. If $b_r<e_l<d_r$, then $f_v<b_r<e_l<d_r<f_r<e_v$, so $e$ is adjacent to $f$, a contradiction. If $e$ intersects $d$ on the left, then the same arguments follow in mirror. Therefore, no MUVI representation exists for the $5$-lobster.

\begin{proposition}
The class of MUVI graphs is properly contained in the class of UVI graphs. \label{MUVI-UVI}
\end{proposition}

\begin{proof}
All MUVI graphs are UVI graphs by definition.  However, Lemma~\ref{5-lobster-lemma} shows that the $5$-lobster is not a MUVI graph, and Figure~\ref{UVI 5-lobster} shows a UVI representation of the $5$-lobster.

\begin{figure}[ht]
\begin{center}
\includegraphics[width=5in]{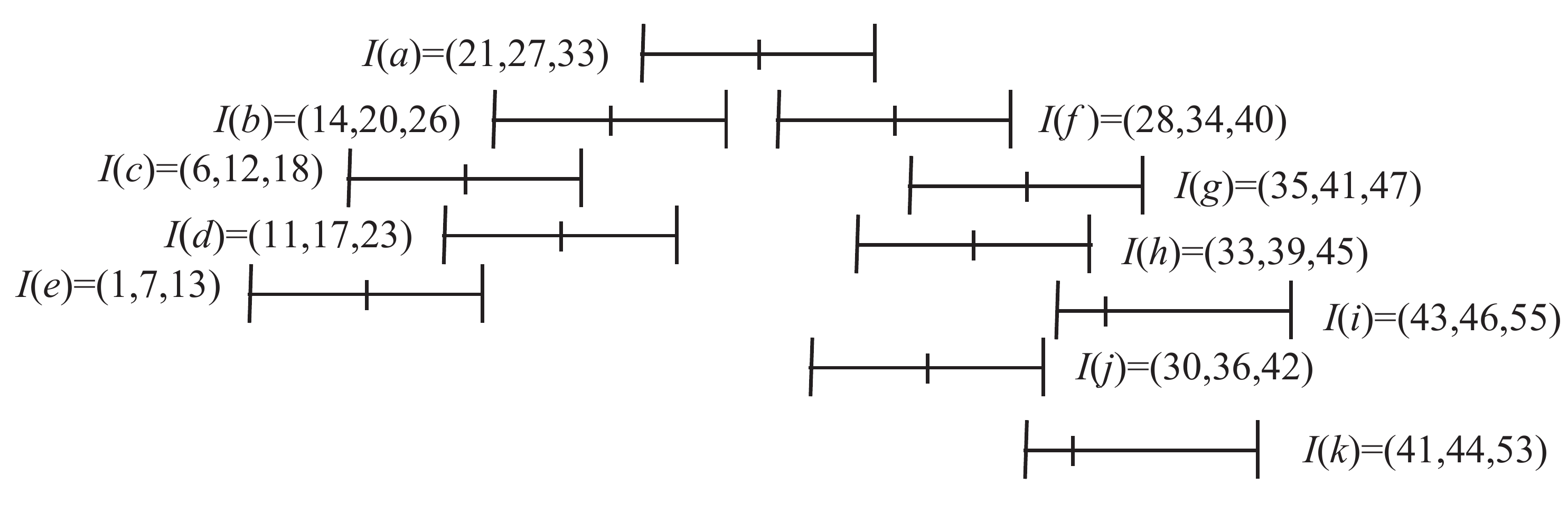}
\end{center} \caption{A UVI representation of the 5-lobster.} \label{UVI 5-lobster}
\end{figure}

\end{proof}

\begin{proposition} \label{PMVI-ordering}
Given a MPVI representation $R$, the left endpoints, right endpoints, and veto marks in $R$ appear in the same order. In other words, for any two intervals $I(a)$ and $I(b)$ in $R$, if $a_1 < b_l$ then $a_v < b_v$ and $a_r<b_r$.
\end{proposition}

\begin{proof}
Consider two intervals $I(a)$ and $I(b)$ in $R$ with $a_l < b_l$.  Since $R$ is proper, $a_r < b_r$.  We will show that $a_v < b_v$.  Denote the intervals $(a_l,a_v)$ and $(a_v,a_r)$ by $AL$ and $AR$, respectively.  Note that since $R$ is a midpoint representation, $AL$ and $AR$ have the same length, and $BL$ and $BR$ have the same length.

Assume by way of contradiction that $b_v < a_v$.  So we have $BL \subset AL$ and $AR \subset BR$. This contradicts the fact that $AL$ and $AR$ have the same length, and $BL$ and $BR$ have the same length.
\end{proof}

\begin{proposition}
A graph $G$ is a PVI graph if and only if $G$ is a UVI graph.
\end{proposition}

\begin{proof}
Note that if intervals are unit length, no interval can be properly contained in another, so every UVI graph is a PVI graph.  For the converse, we follow the proof technique of Bogart and West in their proof that $G$ is a proper interval graph if and only if $G$ is a unit interval graph \cite{Bogart99}.  Note that if no interval in $R$ is properly contained in another, then the left endpoints in $R$ appear in the same order as the right endpoints.  We order the intervals in $R$ in the order of their left endpoints, which is also the order of their right endpoints.  We start with a proper interval representation $R_0$ of $G$ and iteratively construct a new representation $R_i$ in which the first $i$ intervals in this ordering are unit length.

We consider the representation $R_{i-1}$ of $G$ in which the first $i-1$ intervals are unit length, and we adjust the $i$th interval $I=[a,b]$ to be unit length.  In $R_{i-1}$, let $\alpha=a$ if $I$ contains no right endpoints, and $\alpha=c$ otherwise, where $c$ is the rightmost right endpoint contained in $I$.  Note that the interval with right endpoint $c$ is unit length in $R_{i-1}$, so $\alpha<a+1$ and $\alpha<b$.  We transform $R_{i-1}$ into $R_i$ by uniformly shrinking or expanding the part of $R_{i-1}$ in the interval $[\alpha, b]$ to $[\alpha, a+1]$, and translating the part of $R_{i-1}$ in the interval $[b, \infty)$ to $[a+1, \infty)$.  Since this transformation preserves the order of the marked points in $R_{i-1}$, $R_i$ has the same adjacencies as $R_{i-1}$ and also has interval $I$ with unit length.

Hence $R_n$ is a UVI representation of $G$.
\end{proof}

\begin{proposition}
The class of MUVI graphs is properly contained in the class of MPVI graphs. \label{MUVI-MPVI}
\end{proposition}

\begin{proof}
Again, if the intervals are unit length, then no interval is properly contained in another.  So every MUVI representation is also a MPVI representation. Conversely, Figure~\ref{MPVI 5-lobster} shows a MPVI representation of the 5-lobster, which is not a MUVI graph by Lemma~\ref{5-lobster-lemma}.

\begin{figure}[ht]
\begin{center}
\includegraphics[width=\textwidth]{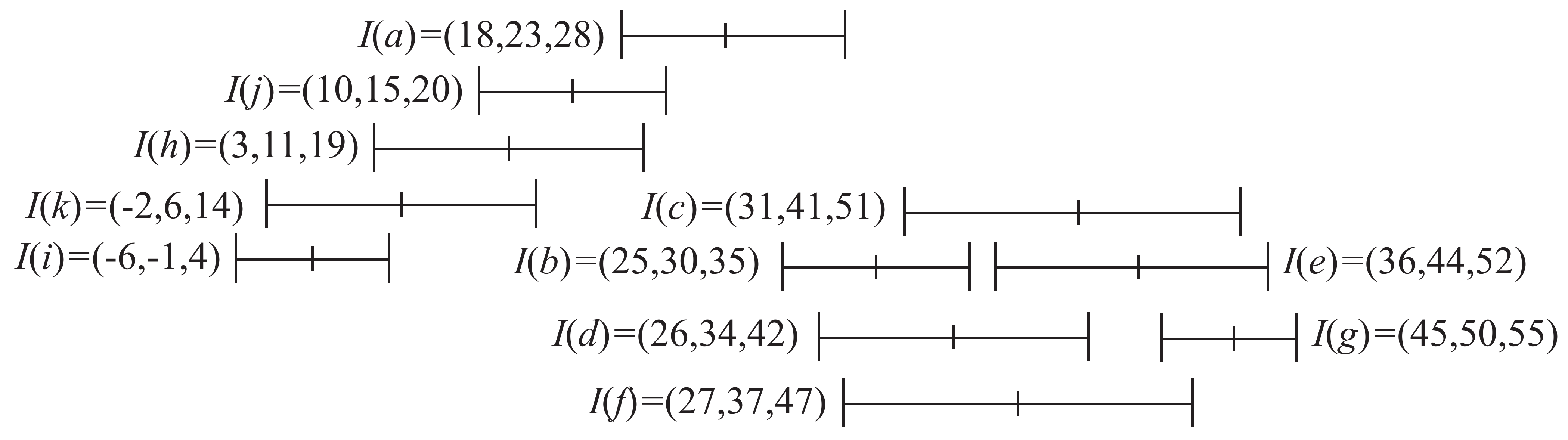}
\end{center} \caption{A MPVI representation of the 5-lobster.} \label{MPVI 5-lobster}
\end{figure}
\end{proof}

\section{Coloring Veto Interval Graphs}

We ask how large the chromatic number of VI graphs can be.  Since VI graphs are triangle-free, examples of VI graphs with large chromatic number are relatively hard to find.  By Proposition~\ref{VI families}, odd cycles are VI graphs, so there are examples of VI graphs with chromatic number 3.  On the other hand, by Theorem~\ref{G9}, not all bipartite graphs are VI graphs.  The Gr\"{o}tszch graph is an example of a triangle-free graph with chromatic number 4, but Theorem~\ref{grotzsch-not-VI} shows that this graph is not a VI graph.  Proposition~\ref{circ} below shows that the circulant graph $\Circ(13, \{1,5\})$ is a 4-chromatic veto interval graph.

\begin{figure}[ht]
\begin{center}
\includegraphics[width=\textwidth]{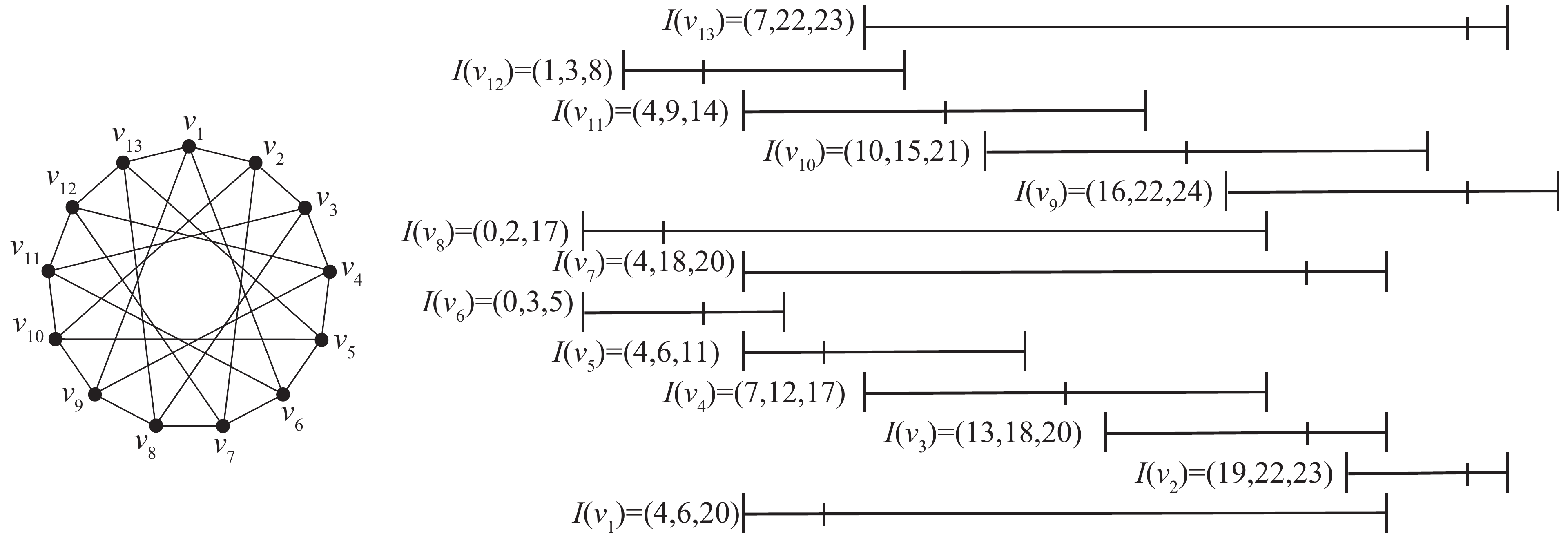}
\end{center} \caption{The circulant graph $\Circ(13, \{1,5\})$ and its veto interval representation.} \label{circulant}
\end{figure}

\begin{proposition}
The circulant graph $\Circ(13, \{1,5\})$, shown in Figure~\ref{circulant}, is a 4-chromatic veto interval graph. \label{circ}
\end{proposition}

\begin{proof}
Heuberger proved that the chromatic number of $\Circ(13, \{1,5\})$ is 4 \cite{Heuberger03}.  A veto interval representation of $\Circ(13, \{1,5\})$ is given in Figure~\ref{circulant}.
\end{proof}

This example was found via a computer search on a collection of 4-chromatic graphs from the House of Graphs website \cite{house-of-graphs}.  A similar search on 5-chromatic graphs has not found any 5-chromaitc veto interval graphs.  We don't know whether there exists a VI graph with chromatic number greater than 4, or indeed any upper bound on the chromatic number of VI graphs.  If the veto intervals are unit length, however, we obtain the following result.

\begin{theorem}
UVI graphs are $4$-colorable. \label{UVI-4-colorable}
\end{theorem}

\begin{proof}
Let $G$ be a UVI graph with UVI representation $R$.  For each interval $I(v) \in R$, color $I(v)$ red if $\lfloor v_l \rfloor$ is odd and $ \lfloor  v_v \rfloor$ is odd, color $I(v)$ blue if $\lfloor v_l \rfloor$ is odd and $ \lfloor  v_v \rfloor$ is even, color $I(v)$ purple if $\lfloor v_l \rfloor$ is even and $ \lfloor  v_v \rfloor$ is even, and color $I(v)$ orange if $\lfloor v_l \rfloor$ is even and $ \lfloor  v_v \rfloor$ is odd.  We show that this is a proper 4-coloring of $G$ by verifying that any two intervals in $R$ with the same color are not adjacent.

Suppose first that $I(v)$ and $I(w)$ have the same color and $\lfloor v_l \rfloor \neq \lfloor w_l \rfloor$.  Since $I(v)$ and $I(w)$ have the same color,  $\lfloor v_l \rfloor$ and $\lfloor w_l \rfloor$ have the same parity, so $I(v)$ and $I(w)$ differ by more than 1.  Since $I(v)$ and $I(w)$ are unit length, they do not overlap.

Now suppose that $I(v)$ and $I(w)$ have the same color and $\lfloor v_l \rfloor = \lfloor w_l \rfloor = k$ for some integer $k$.  If $I(v)$ and $I(w)$ are red or purple, then  $v_l$, $v_v$, $w_l$, and $w_v$ all have the same parity, and since $I(v)$ and $I(w)$ are unit length, $v_l$, $v_v$, $w_l$, and $w_v$ are all contained in the same unit interval $[k, k+1)$, and one must veto the other.  If $I(v)$ and $I(w)$ are blue or orange, and if without loss of generality $v_v < w_v$, then we also have $w_l < v_v$, and $I(v)$ vetoes $I(w)$.

In all cases $I(v)$ and $I(w)$ are not adjacent.
\end{proof}

Let $k$ be the largest chromatic number of any UVI graph.  By Theorems~\ref{MUVI cycle} and \ref{UVI-4-colorable}, $3 \leq k \leq 4$.  The exact value of $k$ is still open.

\section{Additional Variations}

\subsection{Double Veto Interval Graphs}

A \textit{\textbf{double veto interval}} $I(a)$ has two veto marks $a_v$ and $a_w$ with $a_l<a_v<a_w<a_r$.  We denote a double veto interval as an ordered quadruple $I(a)=(a_l,a_v,a_w,a_r)$.  A \textit{\textbf{double veto interval representation}} of a graph $G$ is a set of double veto intervals $S$ and a bijection from the vertices of $G$ to the veto intervals in $S$, such that for any two vertices $a$ and $b$ in $G$, $a$ and $b$ are adjacent if and only if either $a_w<b_l<a_r<b_v$ or $b_w<a_l<b_r<a_v$.  In other words, $a$ and $b$ are adjacent if and only if their corresponding intervals intersect and neither contains either veto mark of the other.  A graph $G$ is a \textit{\textbf{double veto interval graph}} if $G$ has a double veto interval representation.  We define $k$-veto graphs and $k$-veto interval representations analogously, with $k$ veto marks in each interval, and denote the veto marks in an interval in such a representation by $v_1$ through $v_k$.

Since the proof of Lemma~\ref{directed-cycles} applies to double interval graphs and $k$-veto interval graphs, these graphs are also triangle-free. Note that every veto interval graph is also a double veto interval graph.  Given a veto interval representation $R$ of a graph $G$, we can construct a double veto interval representation of $G$ by splitting each veto mark in $R$ into two veto marks a small distance apart.  We know of no example of a double veto interval graph that is not a veto interval graph, so these graph classes may be equal.  We do have the following result showing that we do not obtain a different graph class with more than two veto marks.

\begin{proposition}
A graph $G$ is a $k$-veto graph if and only if $G$ is a double veto graph, for $k \geq 2$.
\end{proposition}

\begin{proof}
Given a double veto representation of $G$, we obtain a $k$-veto representation of $G$ by splitting the left veto mark of each interval in $R$ into $k-1$ veto marks a small distance apart.

Conversely, suppose $R$ is a $k$-veto representation of $G$.  Let $R'$ be the double veto representation obtained by removing the veto marks $v_2$ through $v_{k-1}$ from each interval in $R$.  We claim that $R'$ still represents $G$.

Let $I(x)$ be an arbitrary interval in $R$.  We check that for any other interval $I(y)$, $I(x)$ is adjacent to $I(y)$ in $R$ if and only if $I(x)$ is adjacent to $I(y)$ in $R'$.

We consider the different ways $I(x)$ and $I(y)$ can intersect in $R$. If $I(y)$ does not contain any of the veto marks $x_{v_2}$ through $x_{v_{k-1}}$, then removing these veto marks does not affect the adjacency of $I(x)$ and $I(y)$.  If $I(y)$ does contain one of these veto marks, then either $I(y)$ also contains  $x_{v_1}$ or  $x_{v_k}$, or $I(y)$ is contained in $I(x)$.  In either case, $I(x)$ and $I(y)$ are not adjacent in both $R$ and $R'$.
Therefore $R'$ has the same adjacencies as $R$.
\end{proof}

\subsection{Single approval graphs}

Two veto intervals are considered adjacent if their intersection contains neither of their veto marks.  In this section we consider alternative definitions of adjacency: two intervals are adjacent if their intersection contains one or both of their marks.  To help intuition, in this section we call veto intervals and veto marks \textbf{\textit{approval intervals}} and \textbf{\textit{approval marks}}, respectively, and denote the approval mark of the interval $I_a$ by $a_a$.

A \textit{\textbf{single approval interval representation}} (respectively, \textit{\textbf{double approval interval representation}}) has two intervals $I_a$ and $I_b$ adjacent if they intersect and their intersection contains exactly one of their approval marks (respectively, both of their approval marks).  If the intersection contains $a_a$ we say that $I(a)$ \textit{\textbf{approves}} $I(b)$.  A graph $G$ is a \textit{\textbf{single approval interval graph}} if $G$ has a single approval interval representation, and a \textit{\textbf{double approval interval graph}} if $G$ has a double approval interval representation.

Single and double approval graphs are related to tolerance graphs in the following way.  Tolerance graphs are a generalization of interval graphs in which each vertex in a graph $G$ is assigned an interval and a \textit{tolerance} such that $a$ and $b$ are adjacent in $G$ if $I(a)$ and $I(b)$ intersect, and this intersection is at least as large as either the tolerance of $I(a)$ or the tolerance of $I(b)$.  Tolerance graphs were introduced by Golumbic and Monma in \cite{GolumbicMonma} and have been studied extensively since then.  For a thorough treatment of tolerance graphs, see \cite{GolumbicTrenk04} by Golumbic and Trenk.

Specifically, using the terminology of Golumbic and Trenk, a \textit{\textbf{bounded bitolerance representation}} of a graph $G$ has an interval $I_a=[L(v),R(v)]$ and two tolerant points $p(a)$ and $q(a)$ for each vertex $a$ of $G$, such that $p(a)$ and $q(a)$ are both contained in the interior of $I_a$.  For each interval $I_a$, the lengths $p(a)-L(a)$ and $R(a)-q(a)$ are the \textit{\textbf{left tolerance}} and \textit{\textbf{right tolerance}} of $I_a$.  The vertices $a$ and $b$ are adjacent in $G$ if $I_a$ and $I_b$ intersect and their intersection contains points in $[p(a),q(a)]$ or in $[p(b),q(b)]$
(i.e. they intersect by more than their corresponding tolerance).  If all intervals $I_a$ in a bounded bitolerance representation $R$ have $p(a)=q(a)$, then $R$ is a \textit{\textbf{point-core bitolerance representation}} and its corresponding graph $G$ is a \textit{\textbf{point-core bitolerance graph}}.  If you direct the edges of a point-core bitolerance graph you get a \textit{\textbf{point-core bitolerance digraph}}, and if you then remove the loops you get an \textit{\textbf{interval catch digraph}}, which have been studied extensively \cite{Prisner1, Prisner2}.

In the terminology of approval graphs, a point-core bitolerance representation has a marked point $p(a)=q(a)=a_a$, which is both the endpoint of the left and right tolerance of $I(a)$ and the approval mark of $I(a)$, and two intervals $I(a)$ and $I(b)$ are adjacent if they intersect and their intersection contains either one or both of their approval marks.  In this light, we may ask whether the class of point-core bitolerance graphs is equal to either the class of single approval graphs or the class of double approval graphs.  To answer this question, we note that Golumbic et al. showed that $C_n$ is not a tolerance graph for $n \geq 5$ \cite{GolumbicMonmaTrotter}, but by Theorems~\ref{SA families} and \ref{DA families}, $C_n$ is both a single approval and a double approval interval graph for $n \geq 5$.  Thus tolerance graphs and approval graphs are distinct classes of graphs.

To further highlight the similarities between approval and tolerance graphs, let $R$ be a set of closed intervals with middle marks, and let $G$ be the interval graph with interval representation $R$ (ignoring middle marks), $G_1$ be the veto interval graph with veto interval representation $R$, and $G_2$ be the point-core bitolerance graph with point-core bitolerance representation $R$.  Then $G$ is the disjoint union of $G_1$ and $G_2$.  Furthermore, let $G_3$ be the single approval graph and $G_4$ the double approval graph with representation $R$.  Then $G_2$ is the disjoint union of $G_3$ and $G_4$, so $G$ is also the disjoint union of $G_1$, $G_3$, and $G_4$.  Also note that if $R$ is a midpoint unit representation, then $G_3$ is empty.  An example is shown in Figure~\ref{approval-partition-figure}.

\begin{figure}[ht]
\includegraphics[width=\textwidth]{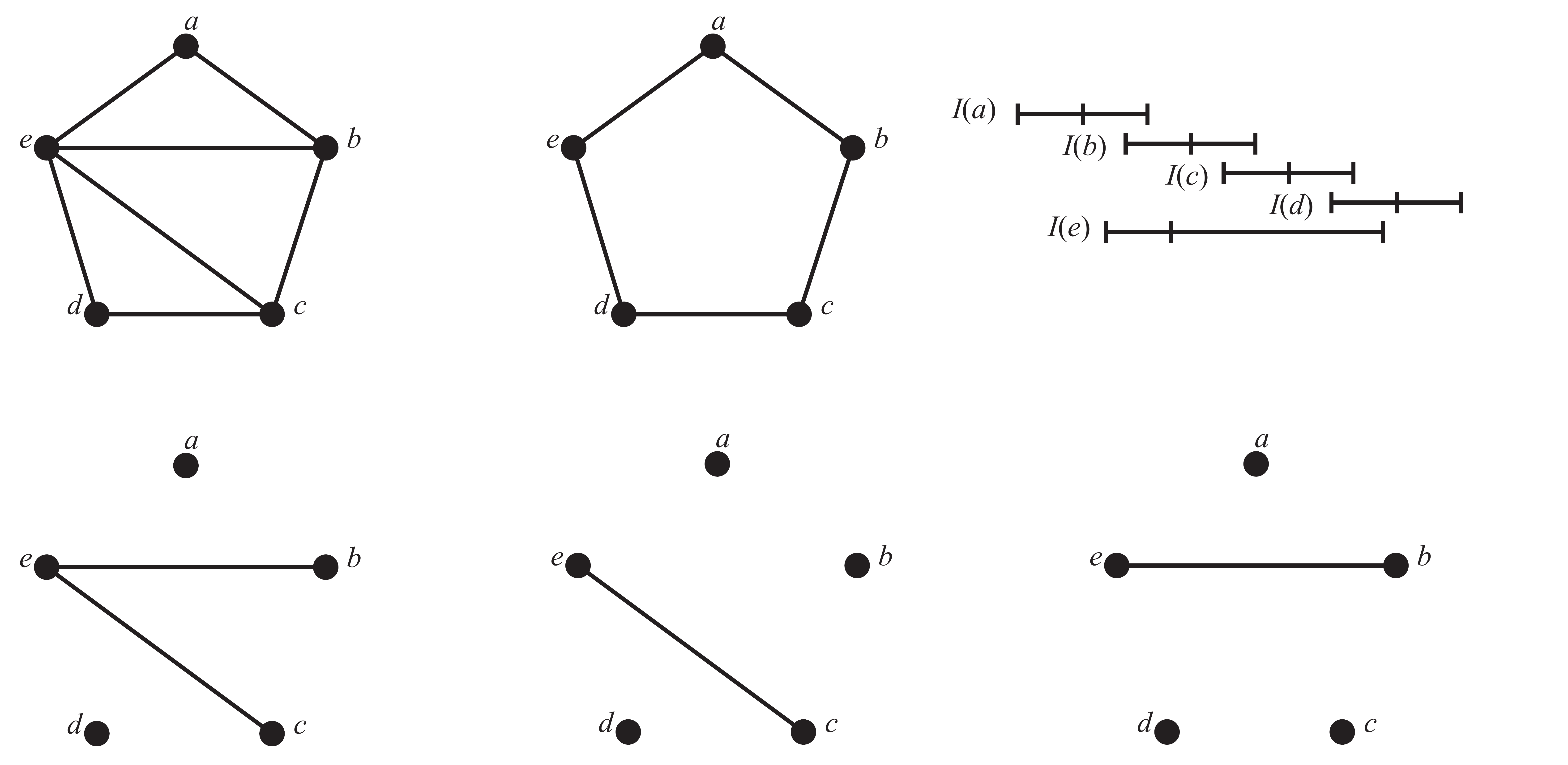}
\caption{A representation and its corresponding interval graph, veto interval graph, point-core bitolerance graph, single approval graph, and double approval graph.} \label{approval-partition-figure}
\end{figure}

We now investigate families of single approval graphs.

\begin{theorem}\label{SA families}
The following families of graphs are single approval interval graphs.
\begin{enumerate}

\item Complete graphs $K_n$, for any positive integer $n$.

\item Cycles $C_n$, for any positive integer $n \geq 3$.

\item Wheels $W_n$, for any positive integer $n \geq 3$.

\item Trees.

\item Complete $k$-partite graphs.

\end{enumerate}
\end{theorem}

\begin{proof}
\begin{enumerate}
\item Label the vertices of $K_n$ as $v_1, v_2, \dots ,v_n$.  To each vertex $v_i$ assign the single approval interval of $(-2i, -2i+1, 2i)$.  This is a single approval interval representation of $K_n$, since each pair of intervals intersect, but only the approval mark of the smaller labeled vertex is contained in the intersection.
\item Label two adjacent vertices of $C_n$ as vertex $v_1$ and $v_2$.  Then label the other vertex adjacent to $v_1$ as $v_3$ and the other vertex adjacent to $v_2$ as $v_4$.  Continue labeling in this way, making $v_5$ the other vertex adjacent to $v_3$ and $v_6$ the other vertex adjacent to $v_4$, until the last vertex gets label $v_n$.  Now assign the single approval interval of $(2,6,7)$ to vertex $v_1$, $(2i, 2i+4, 2i+7)$ to vertex $v_i$, $ 2 \leq i \leq n-1$, and $(2n, 2n+6, 2n+7)$ to vertex $v_n$.  See Figure~\ref{SA_cycleimg} for an example of the vertex labeling and the corresponding single approval interval representation for $C_7$.


\item Label the $n$-cycle of $W_n$ in the same way as $C_n$ was labeled in part $2$, and assign the same single approval intervals, creating a single approval interval representation for this cycle. To the remaining vertex assign the single approval interval $(2, 2n+8, 2n+10)$.  This vertex is adjacent to all of the vertices of the cycle, so its interval overlaps with all of the intervals from the cycle containing exactly one approval mark.

\item We define an algorithm for assigning single approval intervals to the vertices of a tree $T$.  Select a vertex $v_1$ as the root of $T$.  Let $l(v)$ be the distance from a vertex $v$ of $T$ to $v_1$.  Label the remaining vertices of $T$ with $v_2$, $\ldots$, $v_n$ such that if $k<j$ then $l(v_k) \leq l(v_j)$.  We construct a single approval representation $S_i$ of the induced subgraph of $T$ with vertices $\{v_1, \ldots, v_i\}$.  For $S_1$, assign $v_1$ the single approval interval $(0,1,2)$.

Given the single approval representation $S_{i-1}$, we construct the single approval interval $I(v_i)$ to obtain the representation $S_i$.  We define three regions, $R_1$, $R_2$ and $R_3$, in the following way.  Let $v_k$ be the parent of $v_i$.  Now let $R_1$ be the region between the approval mark $v_{k_a}$ and the the marked point immediately to the left of that point in $S_{i-1}$.  Let $R_2$ be the region between $v_{k_r}$ and the first marked point to the right of $v_{k_r}$ in $S_{i-1}$.  If there is no marked point to the right  of $v_{k_r}$ in $S_{i-1}$, then let $R_2$ be the region between $v_{k_r}$ and $v_{k_r}+1$.
Suppose $v_{j_r}$ is the largest right endpoint in $S_{i-1}$ with $l(v_j)<l(v_i)$.  Then we define $R_3$ to be the interval from $v_{j_r}+1$ to $v_{j_r}+2$.

Now we assign the single approval interval of $v_i$ from level $j$ such that $v_{i_l} \in R_1$, $v_{i_a} \in R_2$, and $v_{i_r} \in R_3$, and $v_{i_r}$ appears in the same order among the right endpoints of the vertices in level $j$ as $v_{i_a}$. We also assign marked points to be distinct from previous marked points.

Each pair of vertices within a given level double approve each other, so no adjacencies result.  By construction, each vertex in level $i$ is adjacent to only its parent from level $i-1$.  Lastly, there are no approval marks in any intersection between a vertex in level $i$ and vertices in level $j$, $j < i-1$, so no adjacencies result.

A tree $T$ with corresponding regions $R_1$, $R_2$ and $R_3$ for vertex $v_7$ is shown in Figure~\ref{SA_treeImg}.  The final single approval interval representation of this tree constructed using this technique is given in Figure~\ref{SA_tree_representation}.  More space has been added between marked points in this figure to more easily see the construction.

\item Let $G$ be a complete $k$-partite graph with vertices $v_1$, $v_2$, $\ldots$, $v_n$, and partite sets $\{v_1, \ldots, v_{i_1-1}\}$, $\{v_{i_1}$, $\ldots$, $v_{i_2-1}\}$, $\ldots$, $\{v_{i_{k-1}}$, $\ldots$, $v_n\}$.  Given vertex $v_a$, the $b$th vertex in a partite set with $c$ total vertices, we assign the single approval interval $I(v_a)=(a,n+2a-b,n+2a-b+c)$.  An example showing a single approval interval representation of a complete 3-partite graph is shown in Figure~\ref{SA_kPartiteImg}.
\end{enumerate}
\end{proof}

\begin{figure}[ht]
\begin{center}
\includegraphics[width=1\textwidth]{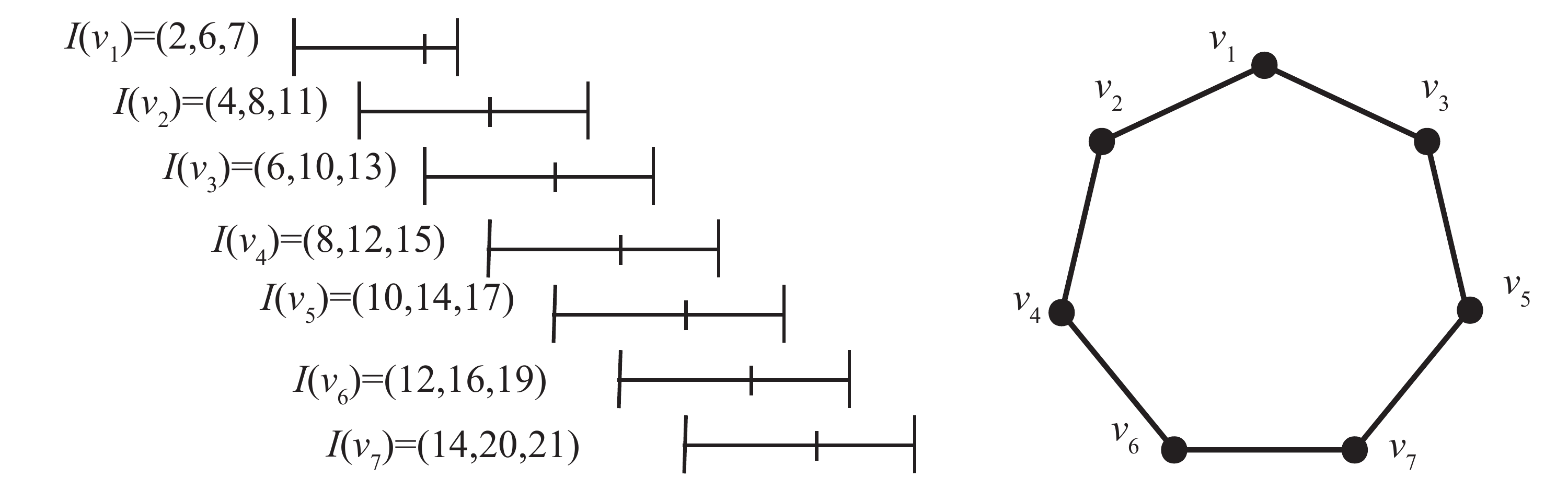}%
\end{center}
\caption{A single approval representation of $C_7$}\label{SA_cycleimg}
\end{figure}

\begin{figure}[ht] \label{SA_wheelImg}
\centerline{%
\includegraphics[width=.5\textwidth]{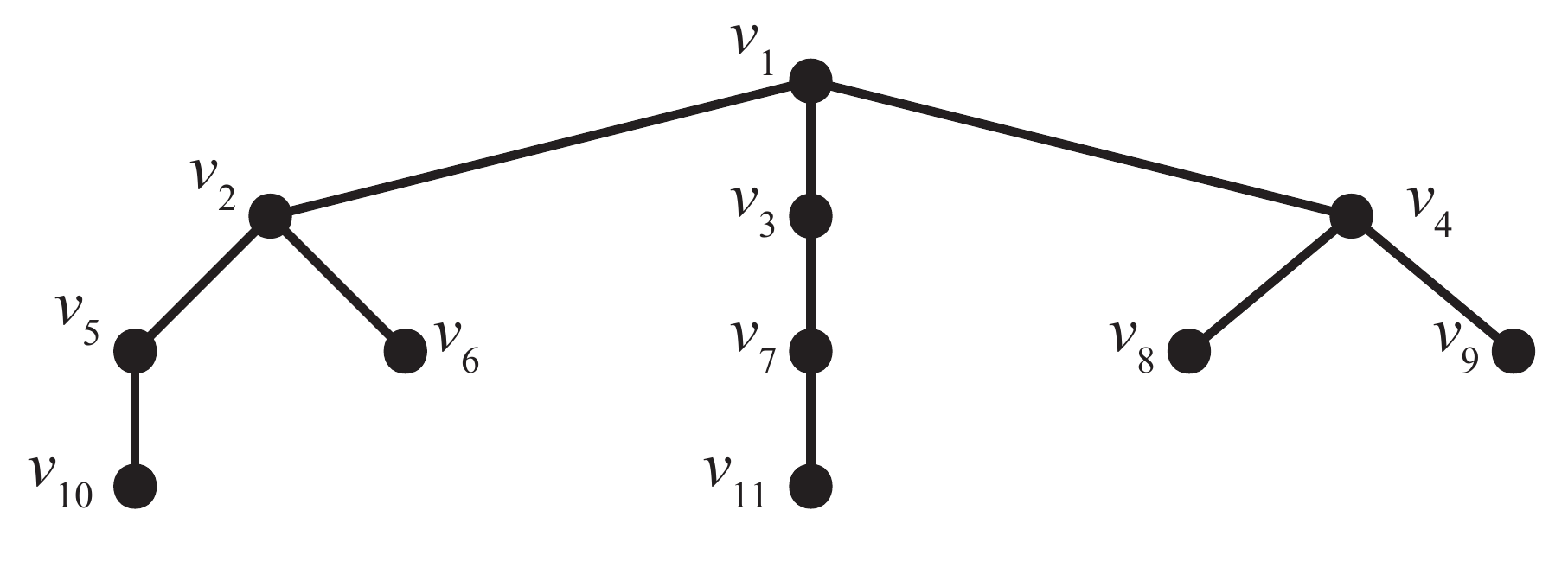} \includegraphics[width=0.5\textwidth]{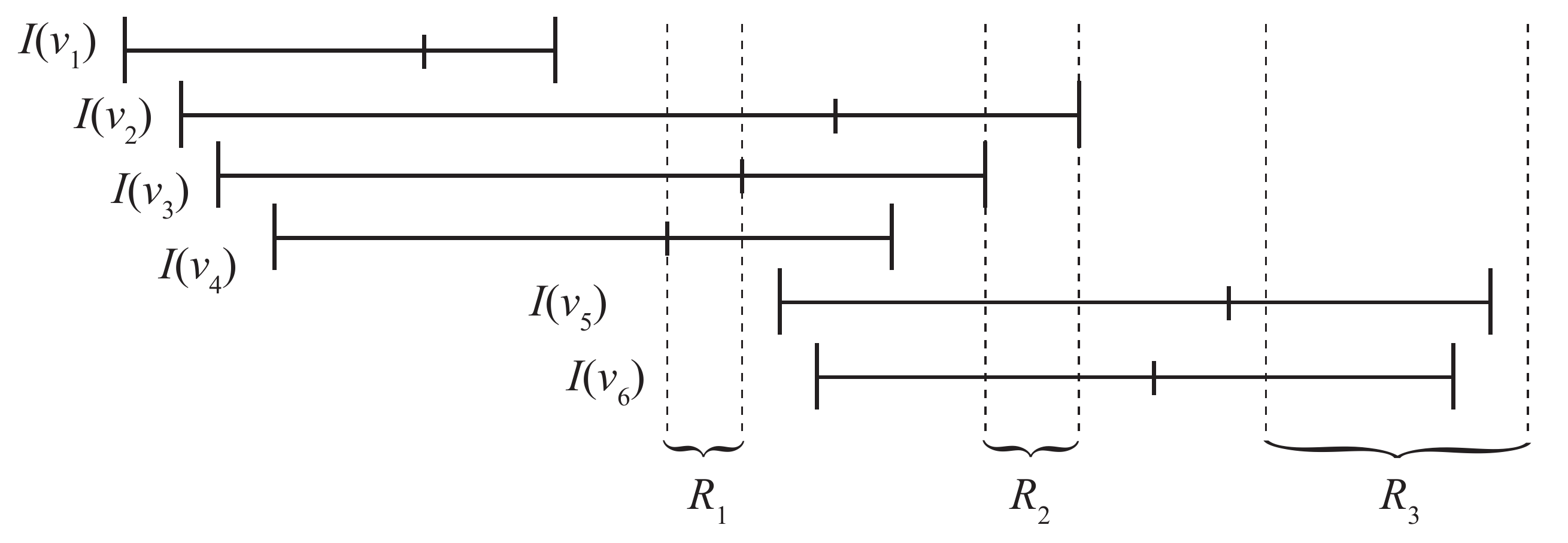}
}
\caption{A tree $T$ and the regions $R_1$, $R_2$, and $R_3$ for vertex $v_7$ in the induction step of part 4 of Theorem~\ref{SA families}.} \label{SA_treeImg}
\end{figure}

\begin{figure}[ht]

\includegraphics[width=\textwidth]{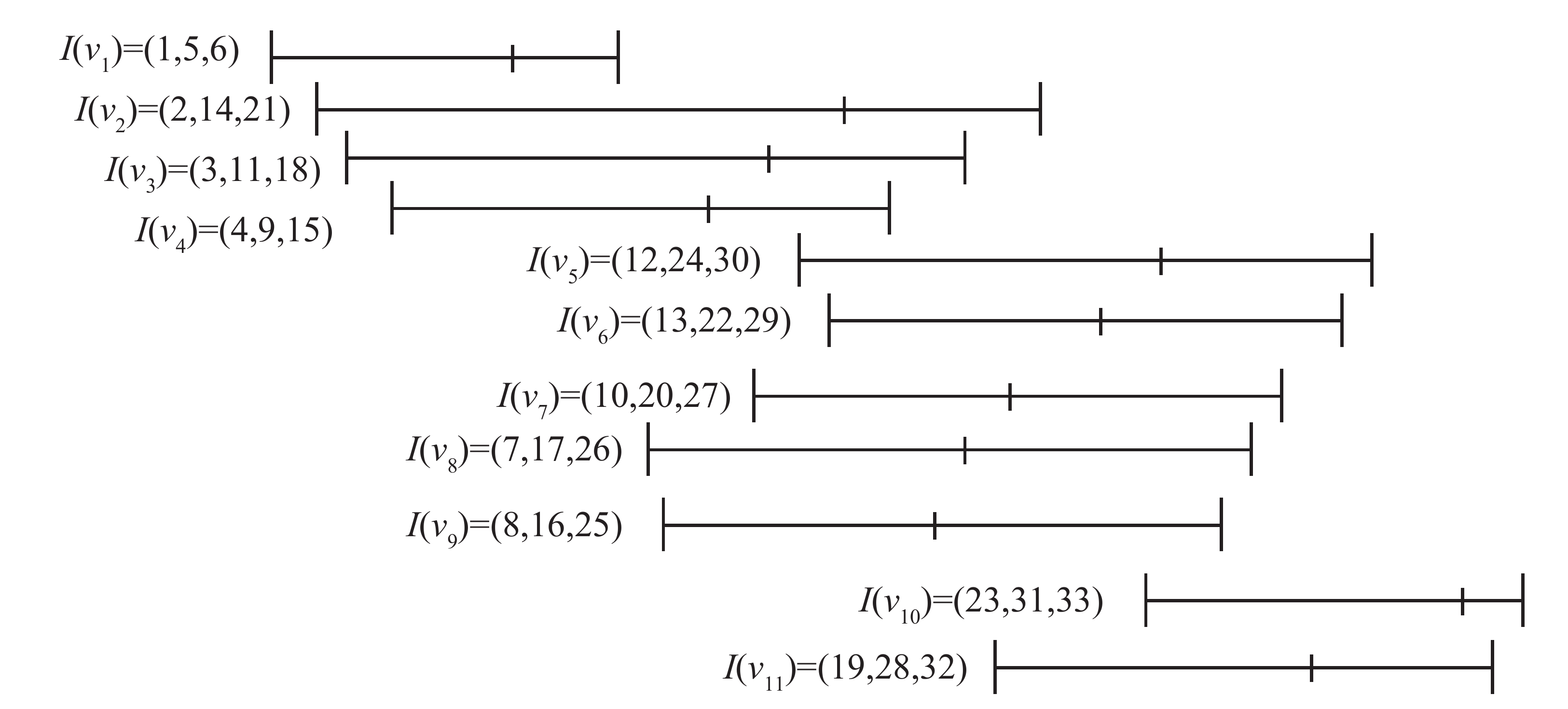}

\caption{A single approval interval representation of the tree in Figure~\ref{SA_treeImg}.}  \label{SA_tree_representation}
\end{figure}

\begin{figure}[ht]
\centerline{%
\includegraphics[width=0.8\textwidth]{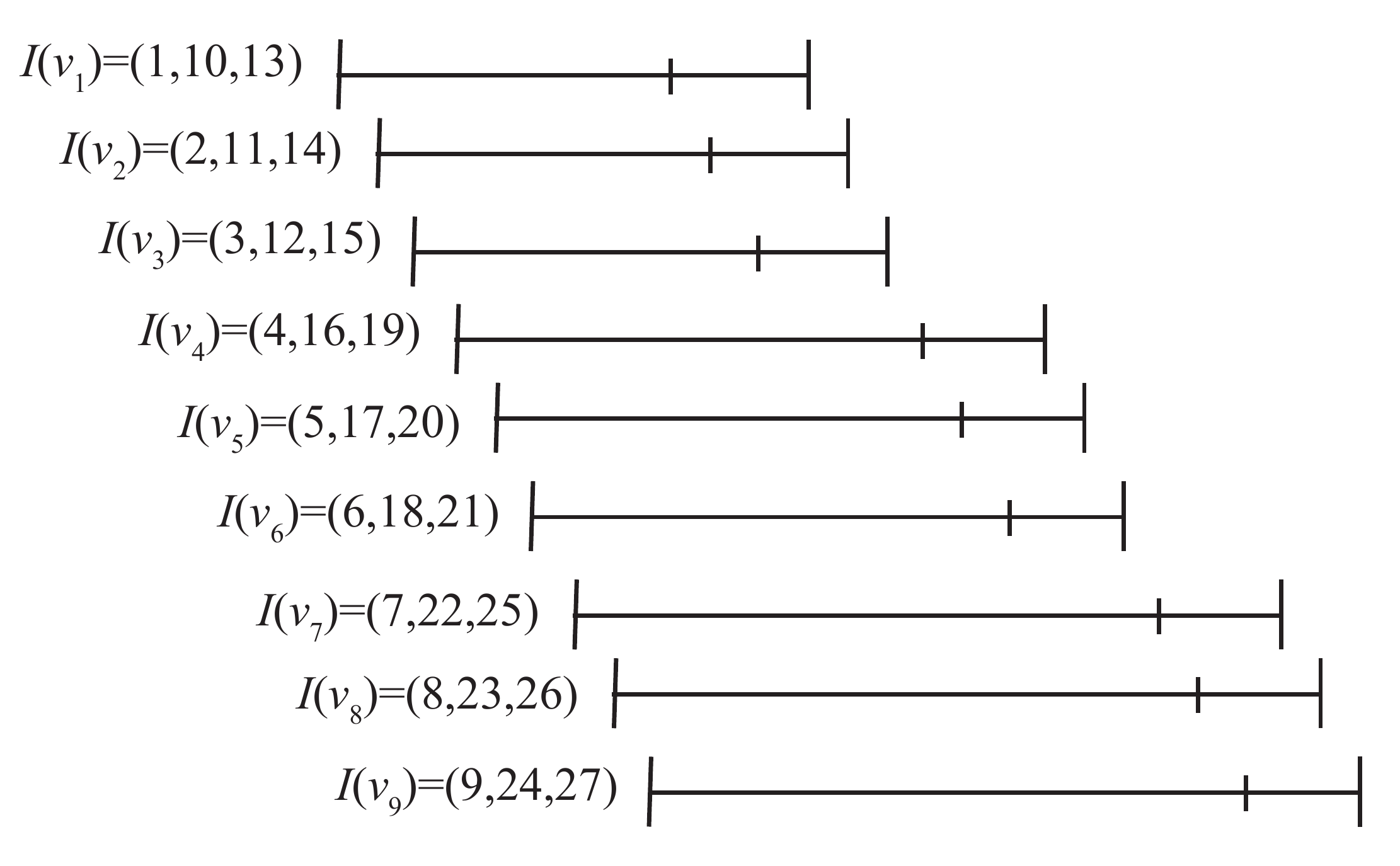}%
}%
\caption{A single approval representation of $K_{3,3,3}$.}  \label{SA_kPartiteImg}
\end{figure}

\begin{theorem}\label{Interval is SA}
The class of interval graphs is properly contained in the class of single approval graphs.
\end{theorem}

\begin{proof}
Consider an interval graph representation.  We can add an approval mark arbitrarily close to the left end point of each interval.  Since whenever two intervals intersect, the left end point of one interval will be contained within the other interval, there will be one approval mark in each intersection.  Also, since all points are distinct in the representation, in an intersection of two intervals there can only be one left end point contained.  Hence, there is exactly one approval mark per intersection of intervals.  Thus all interval graphs are single approval graphs.  Furthermore, interval graphs are a proper subset of single approval graphs since single approval graphs contain cycles which are not contained in interval graphs.
\end{proof}

\subsection{Double approval graphs}

\begin{theorem}\label{DA families}
The following families of graphs are double approval interval graphs.
\begin{enumerate}

\item Complete graphs $K_n$, for any positive integer $n$.

\item Cycles $C_n$, for any positive integer $n \geq 3$.

\item Wheels $W_n$, for any positive integer $n \geq 3$.

\item Complete bipartite graphs $K_{m,n}$, for any positive integers $m$ and $n$.

\item Trees.

\end{enumerate}
\end{theorem}

\begin{proof}
\begin{enumerate}
\item Label the vertices of $K_n$ with $v_1$, $v_2$, $\ldots$, $v_n$.  We let $I(v_i)=(i, i+n, i+2n)$ for $1\leq i \leq n$.  This is a double approval representation of $K_n$.  An example for $n=5$ is shown in Figure~\ref{DA_completeImg}.

\item Since a double approval representation for $C_3=K_3$ is given in part 1, let $n \geq 4$.  Label the vertices of $C_n$ with $v_1$, $v_2$, $\ldots$, $v_n$.  We let $I(v_i)=(i, 2i+n. 2i+n+3)$ for $1 \leq i \leq n-3$, $I(v_{n-2})=(-2, 3n-4, 3n-2)$, $I(v_{n-1})=(-1,0,3n-1)$, and $I(v_n)=(-3,n+1,n+3)$.  An example for $n=6$ is shown in Figure~\ref{DA_cycleImg}.

\item A double approval representation for $W_3=K_4$ is given in part 1.  For the wheel $W_n$ with $n+1$ vertices, $n \geq 4$, add the approval interval $(-4, n, 3n)$ to the representation of $C_n$ in part 2.  An example for $n=7$ is shown in Figure~\ref{DA_completeBipartiteImg}.

\item In $K_{m,n}$, label the vertices in the first partite set with $v_1$, $v_2$, $\ldots$, $v_m$, and the vertices in the second partite set with $w_1$, $w_2$, $\ldots$, $w_n$.  We let $I(v_i)=(2i,2m+2n+2i,2m+2n+2i+1)$ and $I(w_i)=(2m+2i,2m+2i+1,4m+2n+2i)$.  An example is shown in Figure~\ref{DA_completeBipartiteImg}.

\item We consider rooted trees $T$ with root $r$, and for a vertex $v \in T$, let $l(v)$ be its distance from $r$.  Let the height of $T$ be $h=\max(l(v), v\in T)$.  We prove a stronger statement by induction, namely that $T$ has a double approval representation $R$ with $v_l <r_a$ and $v_a>r_a$ if $l(v)=1$ and $v_l>r_a$ if $l(v)>1$.  We induct on the $h$.

For the base case, consider a tree $T$ with $h=1$, i.e. a star with central vertex $r$ and leaves $v_1$ through $v_k$.  Let $I(r)=(0,k+1,3k+2)$ and $I(v_i)=(i,k+2i,k+2i+1)$.

For the induction step, consider a tree $T$ with height $h>1$, and delete its leaves to obtain a tree $T'$ with height $h-1$.  By induction, $T'$ has a double approval interval representation $R'$.  We add intervals to $R'$ for the leaves of $T$ to obtain a double approval interval representation of $T$ in the following way.  For each leaf $x$ of $T'$, let $y_1$, $\ldots$, $y_j$ be the children of $x$ in $T$.  We consider an interval $(x_a-d, x_a+d)$ around the approval mark of $x$ in $R'$ that doesn't contain any other marked points of $R'$.  This can always be done by Lemma~\ref{distinct_points}.  We place the intervals $I(y_1)$, $\dots$, $I(y_j)$ in this interval in a similar way as the base case, by letting $I(y_i)=(x_a-d+id/(2j), x_a+id/(2j), x_a+(2i+1)d/(4j))$, as shown in Figure~\ref{DA-tree-induction}.  Doing this for every leaf of $T'$ yields a double approval interval representation of $T$.
\end{enumerate}
\end{proof}

\begin{figure}[ht]
\begin{center}
\includegraphics[width=0.4\textwidth]{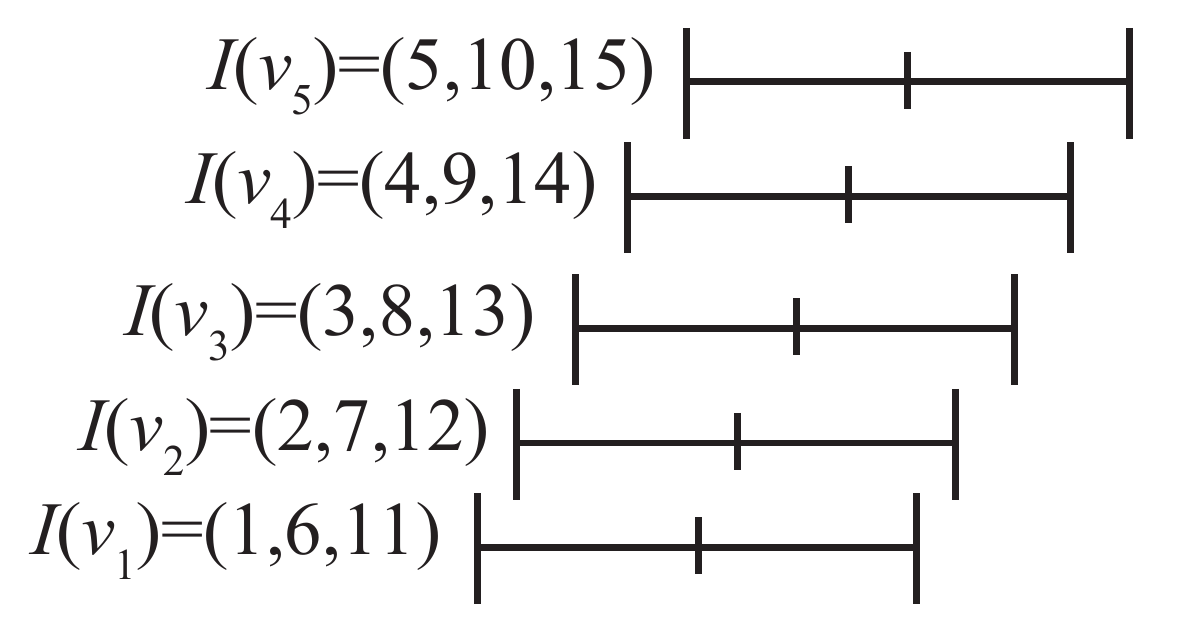}%
\includegraphics[width=0.6\textwidth]{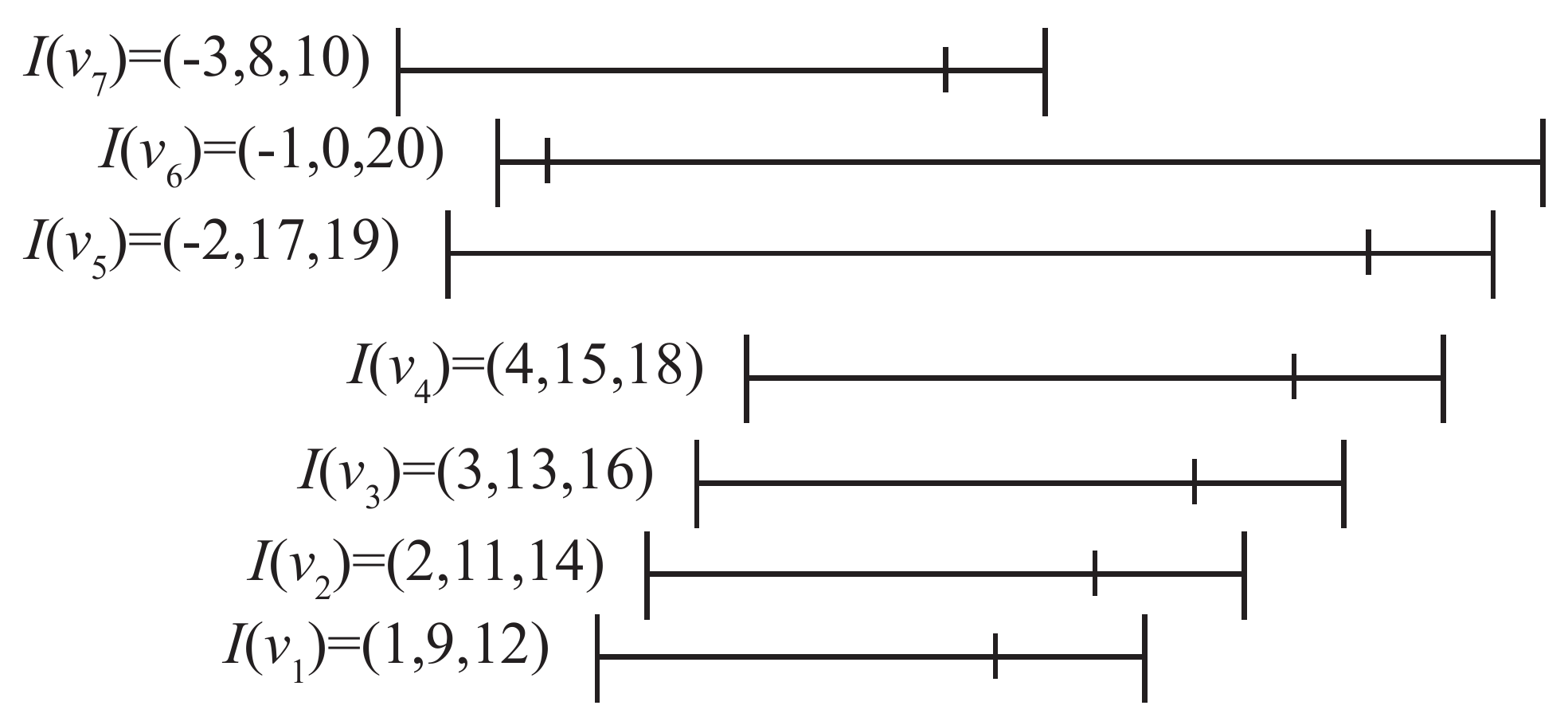}%
\end{center}
\caption{Double approval representations of $K_5$ and $C_7$.}  \label{DA_completeImg} \label{DA_cycleImg}
\end{figure}

\begin{figure}[ht] \label{DA_wheelImg}
\begin{center}
\includegraphics[width=0.47\textwidth]{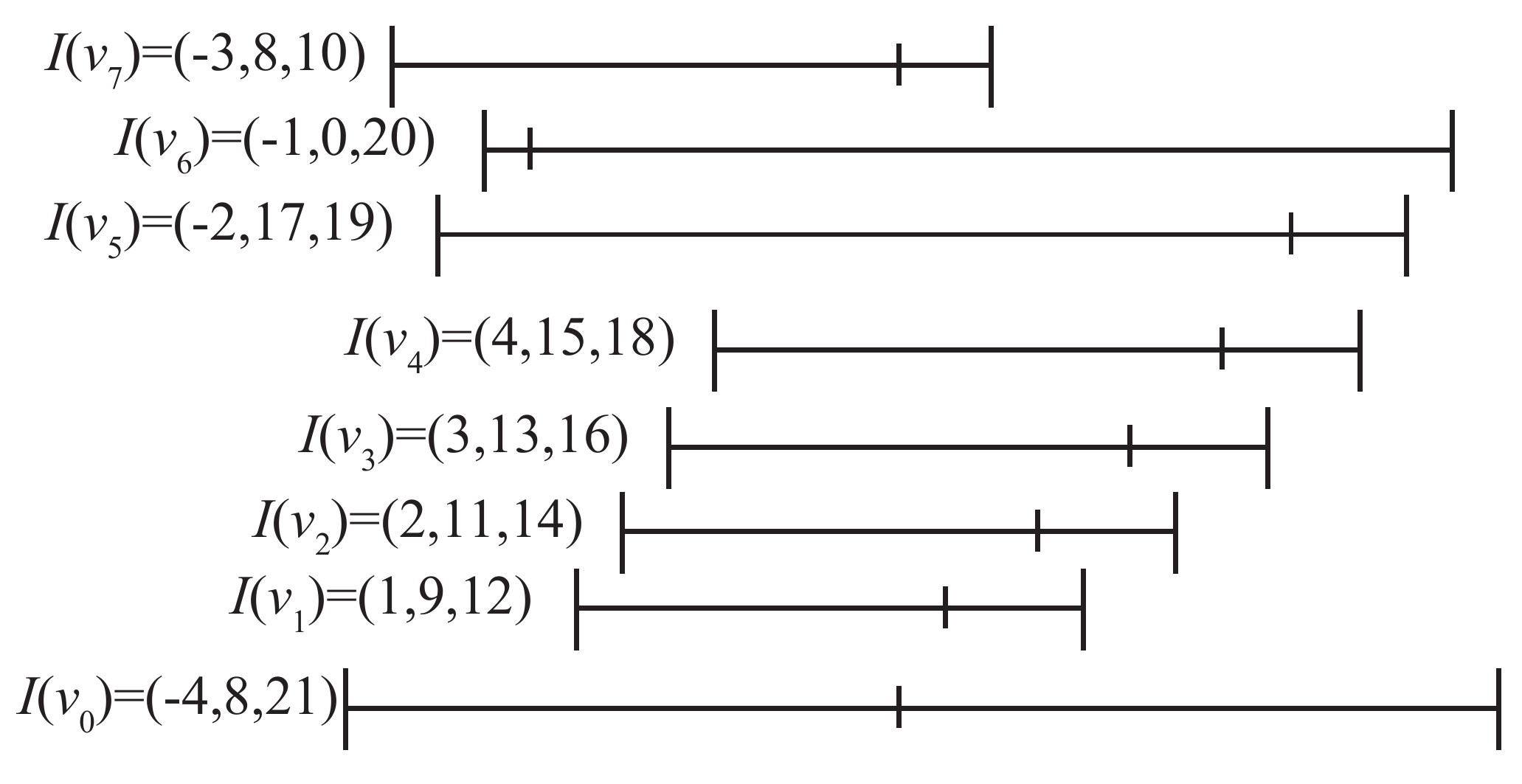} \includegraphics[width=0.47\textwidth]{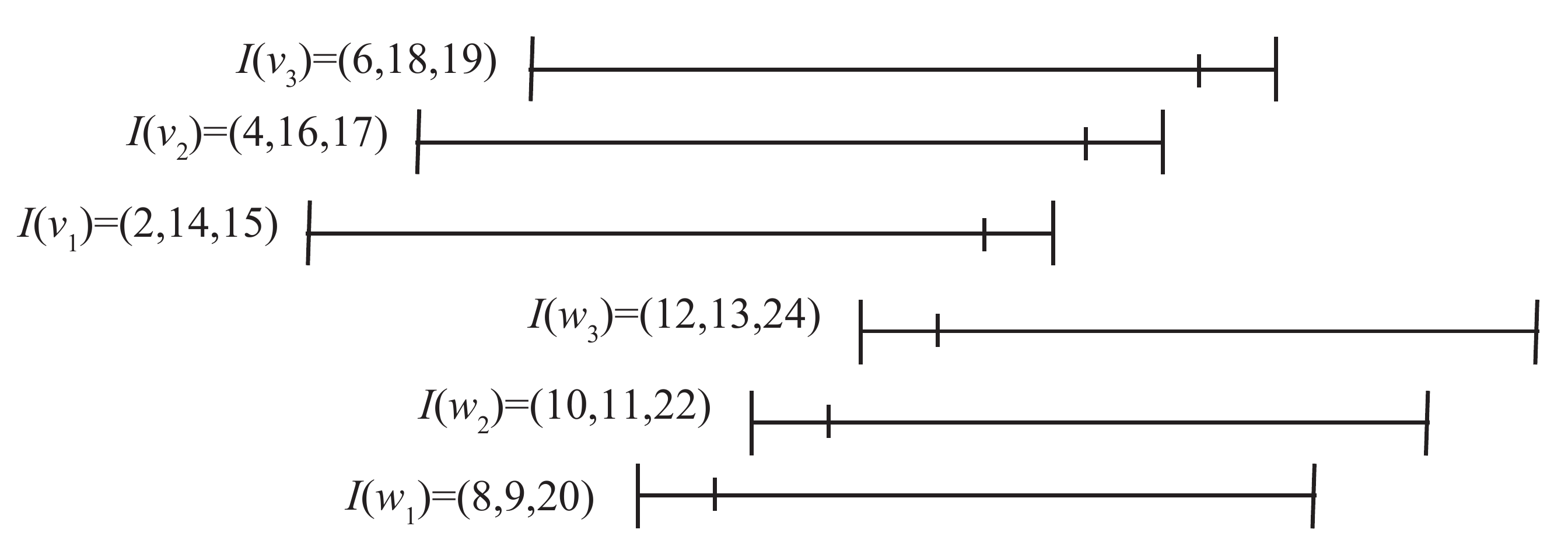}
\end{center}
\caption{Double approval representations of $W_6$ and $K_{3,3}$.} \label{DA_completeBipartiteImg}
\end{figure}

\begin{figure}[ht]
\centerline{%
\includegraphics[width=0.8\textwidth]{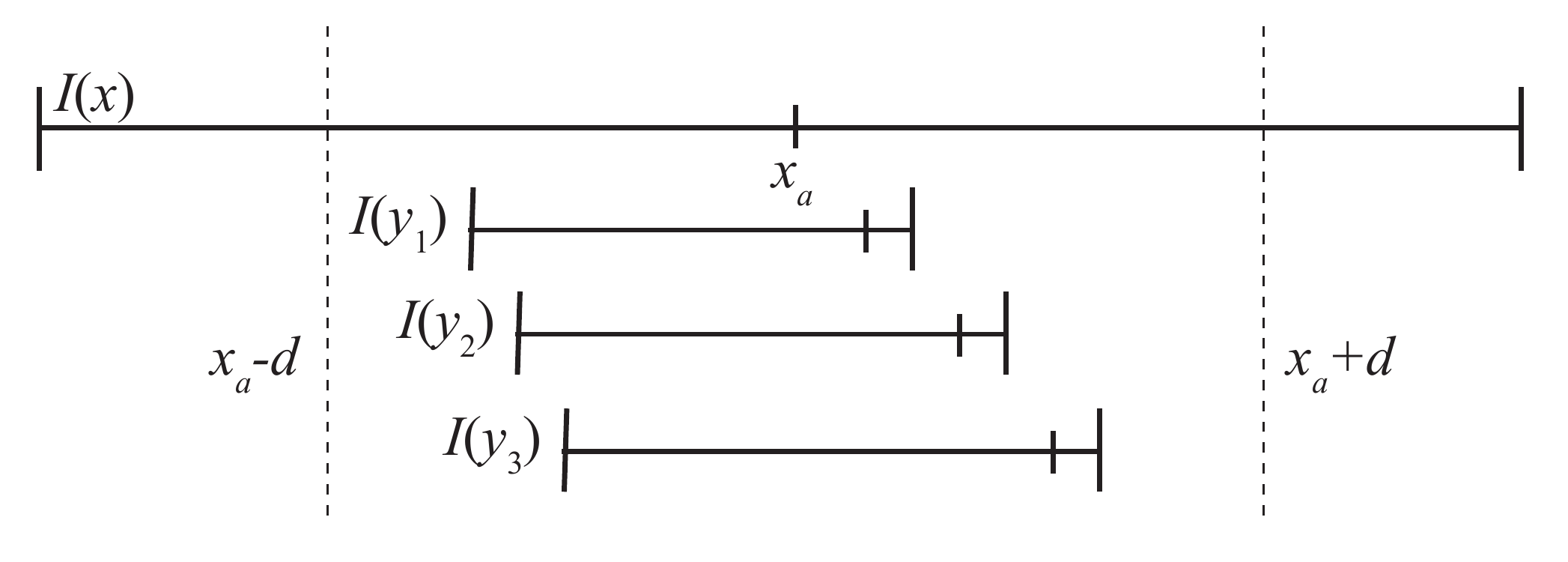}%
}%
\caption{The induction step of Theorem~\ref{DA families} part 5.} \label{DA-tree-induction}

\end{figure}

\begin{proposition}\label{DA not K6-matching}
The tripartite graph $K_{2,2,2}$ is not a double approval graph.
\end{proposition}

\begin{proof}
Suppose $K_{2,2,2}$ has vertices $a$, $b$, $c$, $d$, $e$, and $f$, with missing edges $ab$, $cd$, and $ef$.  Suppose by way of contradiction that $K_{2,2,2}$ has a double approval interval representation $R$.  For each non-adjacent pair of vertices $ab$, $cd$, and $ef$, the approval mark of one of their intervals is outside the other.  Say $a_a$ is outside $I_b$, $c_a$ is outside $I_d$, and $e_a$ is outside $I_f$.  Without loss of generality two of these approval marks are to the left of their corresponding intervals, say $a_a < b_l$ and $c_a < d_l$.  But since $b$ and $c$ are adjacent, $b_l<c_a$.  Therefore $a_a<b_l < c_a< d_l$, contradicting the fact that $a$ and $d$ are adjacent.
\end{proof}

\begin{corollary}
If $a$, $b$, and $c$ are positive integers with $a \leq b \leq c$, then $K_{a,b,c}$ is a double approval interval graph if and only if $a=1$.
\end{corollary}

\begin{proof}
By Proposition~\ref{DA not K6-matching}, if $a \geq 2$ then $K_{a,b,c}$ is not a double approval interval graph.  Conversely, we may add the interval $(1,2m+2n+1.5,4m+4n+1)$ to the representation of $K_{m,n}$ in the proof of Theorem~\ref{DA families} to obtain a double approval interval representation of $K_{1,b,c}$ for all positive integers $b$ and $c$.
\end{proof}

\begin{proposition}\label{MUDA are interval}
A graph $G$ is a midpoint unit double approval interval graph if and only if $G$ is a unit interval graph.
\end{proposition}

\begin{proof}
Let $G$ be a midpoint unit double approval graph, and let $R$ be its midpoint unit double approval interval representation.  We transform $R$ into a unit interval representation $S$ of $G$ as follows.  Suppose every interval in $R$ has length $2c$ for some constant $c$.  Given an approval interval $(a,a+c,a+2c)$ in $R$, define the corresponding interval in $S$ to be $(a+c/2, a+3c/2)$.  Since every interval in $S$ has length $c$, $S$ is a unit interval representation.

We verify that two intervals $I_1=(a,a+c,a+2c)$ and $I_2=(b,b+c,b+2c)$ are adjacent in $R$ if and only if they are adjacent in $S$.  First suppose $I_1$ and $I_2$ are adjacent in $R$, and assume without loss of generality that $a<b$.  Then the approval mark of $I_1$ is contained in $I_2$, so $b<a+c<b+2c$.  Hence $a<b<a+c$, so $a+c/2<b+c/2<a+3c/2$, so $I_1$ and $I_2$ intersect in $S$.  Conversely, if $I_1$ and $I_2$ intersect in $S$ and $a<b$, then $a+c/2<b+c/2<a+3c/2$, and by an analogous argument $I_1$ and $I_2$ are adjacent in $R$.

Since this construction is invertible, we can similarly transform a unit interval representation $S$ of $G$ into a midpoint unit double approval interval representation $R$ of $G$.
\end{proof}

\section{Open Questions}

We conclude with a list of open questions.

\begin{enumerate}

\item Theorem~\ref{G9} shows that $G_{10}$ is a bipartite graph with 56 vertices which is not a VI graph.  What is the smallest bipartite non-VI graph?

\item Does there exist a VI graph with chromatic number 5?  More generally, what is the maximum chromatic number of VI graphs?

\item What is the maximum chromatic number of UVI, MUVI, or MPVI graphs?

\item Is there a double veto interval graph which is not a VI graph?

\item Is there an interval graph which is not a double approval graph?
\item Is there a VI graph which is not a MVI graph?

\item Theorem~\ref{MUVI cycle} shows that all caterpillars are MUVI graphs, but Proposition~\ref{5-lobster-lemma} shows that lobsters are not MUVI graphs.  Which trees are MUVI graphs?

\item Propositions~\ref{MUVI-UVI} and \ref{MUVI-MPVI} show that the class of MUVI graphs is properly contained in both the class of UVI graphs and MPVI graphs.  Is one of these classes of graphs contained in the other?

\item Several classes of intersection graphs generalize interval graphs to two and more dimensions.  It would be interesting to study any of these classes of graphs, such as rectangle intersection graphs, with veto marks.

\end{enumerate}

\section {Acknowledgements}

We thank Benjamin Reiniger for suggesting the House of Graphs as a source of large chromatic number triangle-free graphs.

\bibliographystyle{plain}
\bibliography{veto-interval}

\end{document}